\documentclass[smallextended,final]{svjour3}
\usepackage[letterpaper,top=2in, bottom=1.5in, left=1in, right=1in]{geometry}

\usepackage{epsfig,epsf,fancybox}
\usepackage{amsmath}
\usepackage{mathrsfs}
\usepackage{amssymb}
\usepackage{amsfonts}
\usepackage{graphicx}
\usepackage{color}
\usepackage[linktocpage,pagebackref,colorlinks,linkcolor=blue,anchorcolor=blue,citecolor=blue,urlcolor=blue,hypertexnames=false]{hyperref}
\usepackage{boxedminipage}
\usepackage{stmaryrd}
\usepackage{multirow}
\usepackage{booktabs}
\usepackage{cite}
\usepackage{float}
\usepackage[ruled,vlined,linesnumbered]{algorithm2e}

\usepackage{xcolor}
\usepackage{tcolorbox}

\newcommand*{\textbox}[2]{\colorbox{#1!30}{\parbox{.9\linewidth}{#2}}}

\usepackage{mathtools}

\usepackage[normalem]{ulem} 

\newcommand{\bm}[1]{\boldsymbol{#1}}

\newcommand{\vb}{{\mathbf{b}}}
\newcommand{\vc}{{\mathbf{c}}}

\newcommand{\ve}{{\mathbf{e}}}

\newcommand{\vh}{{\mathbf{h}}}

\newcommand{\vr}{{\mathbf{r}}}
\newcommand{\vs}{{\mathbf{s}}}

\newcommand{\vu}{{\mathbf{u}}}
\newcommand{\vv}{{\mathbf{v}}}

\newcommand{\vx}{{\mathbf{x}}}
\newcommand{\vy}{{\mathbf{y}}}
\newcommand{\vz}{{\mathbf{z}}}

\newcommand{\vA}{{\mathbf{A}}}
\newcommand{\vB}{{\mathbf{B}}}

\newcommand{\vG}{{\mathbf{G}}}
\newcommand{\vH}{{\mathbf{H}}}
\newcommand{\vI}{{\mathbf{I}}}

\newcommand{\vO}{{\mathbf{O}}}

\newcommand{\vQ}{{\mathbf{Q}}}

\newcommand{\vS}{{\mathbf{S}}}

\newcommand{\vU}{{\mathbf{U}}}
\newcommand{\vV}{{\mathbf{V}}}

\newcommand{\vlam}{{\bm{\lambda}}}
\newcommand{\vpi}{{\bm{\pi}}}
\newcommand{\vtheta}{{\bm{\theta}}}

\newcommand{\vPhi}{{\bm{\Phi}}}
\newcommand{\vPsi}{{\bm{\Psi}}}
\newcommand{\vLam}{{\bm{\Lambda}}}

\newcommand{\cF}{{\mathcal{F}}}

\newcommand{\cI}{{\mathcal{I}}}
\newcommand{\cJ}{{\mathcal{J}}}
\newcommand{\cK}{{\mathcal{K}}}

\newcommand{\cM}{{\mathcal{M}}}

\newcommand{\cO}{{\mathcal{O}}}

\newcommand{\cR}{{\mathcal{R}}}

\newcommand{\cX}{{\mathcal{X}}}

\newcommand{\vareps}{\varepsilon}

\newcommand{\EE}{\mathbb{E}} 
\newcommand{\RR}{\mathbb{R}} 
\newcommand{\ZZ}{\mathbb{Z}} 
\renewcommand{\SS}{{\mathbb{S}}} 
\newcommand{\vzero}{\mathbf{0}} 
\newcommand{\vone}{{\mathbf{1}}} 

\newcommand{\prox}{{\mathbf{prox}}} 
\newcommand{\Proj}{{\mathrm{Proj}}} 

\newcommand{\st}{\mbox{ s.t. }}

\DeclareMathOperator*{\argmin}{arg\,min} 

\newcommand{\bc}{\begin{center}}
\newcommand{\ec}{\end{center}}

\newcommand{\bdm}{\begin{displaymath}}
\newcommand{\edm}{\end{displaymath}}

\newcommand{\beq}{\begin{equation}}
\newcommand{\eeq}{\end{equation}}

\newcommand{\bfl}{\begin{flushleft}}
\newcommand{\efl}{\end{flushleft}}

\newcommand{\bt}{\begin{tabbing}}
\newcommand{\et}{\end{tabbing}}

\newcommand{\beqn}{\begin{eqnarray}}
\newcommand{\eeqn}{\end{eqnarray}}

\newcommand{\beqs}{\begin{align*}} 
\newcommand{\eeqs}{\end{align*}}  

\usepackage{braket,mathdots}
\usepackage{mathtools}
\def\R{\RR}
\def\uu{\bar{\vu}}
\def\xu{\bar{\vx}}
\def\yu{\bar{\vy}}

\makeatletter
\newcommand*{\textlabel}[2]{%
	\edef\@currentlabel{#1}
	\phantomsection
	#1\label{#2}
}
\makeatother
\def\Span{\ensuremath{\operatorname{span}}}

\newtheorem{assum}{Assumption}
\def\vgap{\vspace*{.1in}}

\begin{document}

\title{Lower complexity bounds of first-order methods for convex-concave bilinear saddle-point problems}

\author{Yuyuan Ouyang \and Yangyang Xu}

\institute{Y. Ouyang \at Department of Mathematical Sciences, Clemson University, Clemson, SC\\
\email{yuyuano@clemson.edu}\\
Y. Xu \at Department of Mathematical Sciences, Rensselaer Polytechnic Institute, Troy, NY\\
\email{xuy21@rpi.edu}}

\date{\today}

\maketitle

\begin{abstract}
On solving a convex-concave bilinear saddle-point problem (SPP), there have been many works studying the complexity results of first-order methods. These results are all about upper complexity bounds, which can determine at most how many efforts would guarantee a solution of desired accuracy. In this paper, we pursue the opposite direction by deriving lower complexity bounds of first-order methods on large-scale SPPs. Our results apply to the methods whose iterates are in the linear span of past first-order information, as well as more general methods that produce their iterates in an arbitrary manner based on first-order information.
We first work on the affinely constrained smooth convex optimization that is a special case of SPP. Different from gradient method on unconstrained problems, we show that first-order methods on affinely constrained problems generally cannot be accelerated from the known convergence rate $O(1/t)$ to $O(1/t^2)$, and in addition, $O(1/t)$ is optimal for convex problems. Moreover, we prove that for strongly convex problems, $O(1/t^2)$ is the best possible convergence rate, while it is known that gradient methods can have linear convergence on unconstrained problems. Then we extend these results to general SPPs. It turns out that our lower complexity bounds match with several established upper complexity bounds in the literature, and thus they are tight and indicate the optimality of several existing first-order methods.
\\

\noindent {\bf Keywords} Convex optimization, saddle point problems, first-order methods, information-based complexity, lower complexity bound 
\vspace{0.3cm}

\noindent {\bf Mathematics Subject Classification} 90C25, 90C06, 90C60, 49M37, 68Q25

\end{abstract}

\section{Introduction}
In recent years, first-order methods have been particularly popular partly due to the huge scale of many modern applications. These methods only access the function value and gradient information of the underlying problems, and possibly also other ``simple'' operations. For example, on solving the constrained optimization problem $f^*:=\min_{\vx\in X} f(\vx)$, the projected gradient (PG) method 
$$\vx^{(t+1)}\gets \Proj_X(\vx^{(t)} - \alpha \nabla f(\vx^{(t)}))$$ is a first-order method if $\Proj_X$ is easy to evaluate such as projection onto a box constraint set. For convex problems, if $\nabla f$ is Lipschitz continuous and $\alpha$ is appropriately chosen, the PG method can have convergence rate in the order of $\frac{1}{t}$, namely, $f(\vx^{(t)})-f^* = O(\frac{1}{t})$, where $t$ is the number of gradient evaluations. Through smart extrapolation, the rate can be improved to $O(\frac{1}{t^2})$; see \cite{FISTA2009, nesterov2013gradient}. In addition, there exists an instance showing that the order $\frac{1}{t^2}$ cannot be further improved \cite{nesterov2004introductory} and thus is optimal.

In this paper, we consider the bilinear saddle-point problem (SPP):
\begin{equation}\label{eq:saddle-prob}
\min_{\vx\in X}\max_{\vy\in Y} f(\vx)+\langle \vA\vx-\vb, \vy\rangle - g(\vy).
\end{equation}
Here, $X\subseteq\RR^n$ and $Y\subseteq\RR^m$ are closed convex sets,  $f:\RR^n\to\RR$ and $g:\RR^m\to\RR$ are closed convex functions, $\vA\in\RR^{m\times n}$, and $\vb\in\RR^m$. We assume that 
the function $f$ is differentiable, and  $\nabla f$ is Lipschitz continuous with respect to a norm $\|\cdot\|$, namely, there is a constant $L_f>0$ such that
\begin{align}
\label{eq:L}
\|\nabla f(\vx_1) - \nabla f(\vx_2)\|_*\le L_f\|\vx_1-\vx_2\|,\, \forall\, \vx_1,\vx_2\in X,
\end{align}
where $\|\cdot\|_*$ denotes the dual norm of $\|\cdot\|$. In addition, we assume that $g$ is simple such that its proximal mapping can be easily computed. The scale of the problem is large, and it is expensive to form the Hessian of $f$ and solve or project onto a linear system of size $m\times n$.  Two optimization problems are associated with \eqref{eq:saddle-prob}. One is called the primal problem
\begin{align}
\label{eq:SPP}
\phi^*:=\min_{\vx\in X}\Set{\phi(\vx): = f(\vx) + \max_{\vy\in Y}\,\langle \vA\vx-\vb, \vy\rangle - g(\vy)},
\end{align}
and the other is the dual problem
\begin{align}
\label{eq:SPP-dual}
\psi^*:=\max_{\vy\in Y}\Set{\psi(\vy): = - g(\vy)+\min_{\vx\in X}\, \langle \vA\vx-\vb, \vy\rangle +f(\vx)}.
\end{align}
The weak duality always holds, i.e., 
\begin{equation}\label{eq:weak-duality}
\psi^*\le\phi^*.
\end{equation}
Under certain mild assumptions (e.g., $X$ and $Y$ are compact \cite{rockafellar2015convex}), the above inequality becomes an equality, namely, the strong duality holds.

Many applications can be formulated into an SPP. For instance, it includes as special cases all affinely constrained smooth convex optimization problems. To see this, let $Y=\RR^m$ and $g\equiv 0$. Then $\max_\vy\langle \vA\vx-\vb, \vy\rangle=0$ if $\vA\vx=\vb$ and $\infty$ otherwise, and thus \eqref{eq:SPP} becomes 
	\begin{align}
	\label{eq:ECO}
	f^*:=\min_{\vx\in X} \big\{f(\vx), \st \vA\vx = \vb\big\}.
	\end{align}

\subsection{Main goal}
We aim at answering the following question:
\begin{center}
\textbox{gray}{For any first-order method, what is the best possible performance on solving a general SPP?}
\end{center}
More precisely, our goal is to study the \emph{lower information-based complexity bound} of \emph{first-order methods} on solving the class of problems that can be formulated into \eqref{eq:saddle-prob}. In the literature, all existing works about first-order methods on solving saddle-point problems only provide upper complexity bounds. Establishing lower complexity bounds is important because they can tell us whether the existing methods are improvable and also because they can guide us to design ``optimal'' algorithms that have the best performance. To achieve this goal, we will construct worst-case SPP instances such that the complexity result of a first-order method to reach a desired accuracy is lower bounded by a problem-dependent quantity. 

In the above question, we say an iterative algorithm for solving \eqref{eq:saddle-prob} is a \emph{first-order method} if it accesses the information of the function $f$ and the matrix $\vA$ through a \emph{first-order oracle}, denoted by ${\cO}:\RR^n\times\RR^m\to\RR^n\times\RR^m\times\RR^n$. For an inquiry on any point $(\vx,\vy)\in \RR^n\times\RR^m$, the oracle returns 
\begin{align}
\label{eq:O}
{\cO}(\vx,\vy):=\big(\nabla f(\vx), \vA\vx, \vA^\top \vy\big).
\end{align}
Given an initial point $(\vx^{(0)},\vy^{(0)})$, a first-order method $\cM$ for solving SPPs, at the $t$-th iteration, calls the oracle on $(\vx^{(t)},\vy^{(t)})$ to collect the oracle information $\cO(\vx^{(t)},\vy^{(t)})$ and then obtains a new point $(\vx^{(t+1)},\vy^{(t+1)})$ by a rule $\cI_t$. The complete method $\cM$ can be described by the initial point $(\vx^{(0)},\vy^{(0)})\in X\times Y$ and a sequence of rules $\{{\cI}_t\}_{t=0}^{\infty}$ such that
\begin{align}
\label{eq:I}
\big(\vx^{(t+1)},\vy^{(t+1)},\xu^{(t+1)},\yu^{(t+1)}\big) = {\cI}_t\left(\vtheta; {\cO}(\vx^{(0)},\vy^{(0)}), \ldots, {\cO}(\vx^{(t)},\vy^{(t)})\right),\ \forall\, t\ge 0,
\end{align}
where $(\vx^{(t)},\vy^{(t)})\in X\times Y$ denotes the inquiry point, and $(\xu^{(t)},\yu^{(t)})\in X\times Y$ is the approximate solution output by the method. 
Note that we allow $\cM$ to possess the complete knowledge of the rest information in an SPP, e.g., the sets $X$ and $Y$, the function $g$, the vector $\vb$, and the Lipschitz constant $L_f$ and its associated norm. We use $\vtheta$ for
all the rest information. 

As an example of the first-order method in \eqref{eq:I}, we consider \eqref{eq:ECO} and the linearized augmented Lagrangian method (LALM) with iterative updates:
\begin{subequations}\label{eq:update-lalm}
\begin{align}
&\vx^{(t+1)}=\Proj_X\left(\vx^{(t)}-\frac{1}{\eta}\big(\nabla f(\vx^{(t)})+\vA^\top(\vlam^{(t)}+\vr^{(t)})\big)\right),\label{eq:update-lalm-x}\\
&\vlam^{(t+1)} = \vlam^{(t)}+\vr^{(t+1)}, \label{eq:update-lalm-lam}
\end{align}
\end{subequations}
where $\eta>0$ is a stepsize parameter, and $\vr^{(t)}=\vA\vx^{(t)}-\vb$. As a special case, assume that $X=\RR^n$, $\vx^{(0)}=\vzero$ and $\vlam^{(0)}=\vzero$. In such special case it is easy to observe that $\vlam^{(t)}=\sum_{j=1}^t  \vr^{(j)},\,\forall \,t\ge 0$, where we have used the convention $$\sum_{j=j_1}^{j_2} \vr^{(j)}=\vzero,\text{ if }j_1>j_2.$$ Hence, the update in \eqref{eq:update-lalm-x} becomes exactly
\begin{equation}\label{eq:lalm-x}
\vx^{(t+1)}= - \frac{1}{\eta}\left(\sum_{j=0}^t\left(\nabla f(\vx^{(j)})+ \vA^\top\vr^{(j)}\right)+\vA^\top\sum_{j=1}^t  \vr^{(j)} \right),\, \forall\, t\ge 0.
\end{equation} 
 Let $\cO(\vu,\vv) = (\nabla f(\vu), \vA\vu, \vA^\top \vv)$ and $\vu^{(0)}=\vzero$ and $\vv^{(0)}=\vzero$. We define the rules $\{\cI_k\}_{k=0}^\infty$ as follows:
\begin{enumerate}
\item $\cI_0: \big(\vtheta; \cO(\vu^{(0)},\vv^{(0)})\big) \mapsto (\vu^{(1)}, \vv^{(1)}, \uu^{(1)}, \bar{\vv}^{(1)})$ is given by 
$$\uu^{(1)}=\vu^{(1)}=\vzero, \, \bar{\vv}^{(1)}=\vv^{(1)}=\vA\vu^{(0)}-\vb;$$
\item For any $k\ge 1$, $\cI_{2k-1}: \big(\vtheta; \cO(\vu^{(0)},\vv^{(0)}),\ldots,\cO(\vu^{(2k-1)},\vv^{(2k-1)})\big) \mapsto (\vu^{(2k)}, \vv^{(2k)}, \uu^{(2k)}, \bar{\vv}^{(2k)})$ is given by
\begin{align*}
&\uu^{(2k)}=\vu^{(2k)}=-\frac{1}{\eta}\left(\sum_{j=1}^k\left(\nabla f(\vu^{(2j-2)})+\vA^\top\vv^{(2j-1)}\right)+\vA^\top\sum_{j=2}^k\vv^{2j-1}\right), \\
&\bar{\vv}^{(2k)}=\vv^{(2k)}=\vA\vu^{(2k-2)}-\vb;
\end{align*}
\item For any $k\ge 1$, $\cI_{2k}: \big(\vtheta; \cO(\vu^{(0)},\vv^{(0)}),\ldots,\cO(\vu^{(2k)},\vv^{(2k)})\big) \mapsto (\vu^{(2k+1)}, \vv^{(2k+1)}, \uu^{(2k+1)}, \bar{\vv}^{(2k+1)})$ is given by
\begin{align*}
&\uu^{(2k+1)}=\vu^{(2k+1)}=-\frac{1}{\eta}\left(\sum_{j=1}^k\left(\nabla f(\vu^{(2j-2)})+\vA^\top\vv^{(2j-1)}\right)+\vA^\top\sum_{j=2}^k\vv^{2j-1}\right), \\
&\bar{\vv}^{(2k+1)}=\vv^{(2k+1)}=\vA\vu^{(2k)}-\vb.
\end{align*}
\end{enumerate}
In the above defined rules, we update $\vu$ by every odd-numbered rule and $\vv$ by every even-numbered rule. Through comparing the above rules and \eqref{eq:lalm-x}, it is not difficult to verify that $\vu^{(2k)}=\vx^{(k)}$ and $\vv^{(2k+1)}=\vr^{(k)}$ for any $k\ge 0$. Therefore, the LALM can be described as \eqref{eq:I}.

Note that the above description satisfies Assumption \ref{assum:linear_span} below, 
namely, $\vx^{(t+1)}$ in \eqref{eq:lalm-x}  is a linear combination of all previous first-order information $\nabla f(\vx^{(j)})$ and $\vA^\top\vr^{(j)}$. However, for general $X$, the description in \eqref{eq:update-lalm} may not fall into the above linear span description, and we need use the description in \eqref{eq:I}.

Existing works (e.g., \cite{xu2017accelerated, gao2017first}) show that in $t$ iterations, the LALM for \eqref{eq:ECO} can generate an $O(\frac{1}{t})$-optimal solution $\bar{\vx}$, namely, $|f(\bar{\vx})-f(\vx^*)|=O(\frac{1}{t})$ and $\|\vA\bar{\vx}-\vb\|=O(\frac{1}{t})$. In addition, by smoothing technique, \cite{nesterov2005smooth} gives a first-order method for the problem \eqref{eq:SPP} and establishes its $O(\frac{1}{t})$ convergence rate result.  We will show that different from the projected gradient method, the order $\frac{1}{t}$ generally cannot be improved to $\frac{1}{t^2}$, and in addition it is optimal.   

\subsection{Main results}
We consider both convex and strongly convex cases. Here, without specifying details, we state the main results that are obtained in this paper. The following theorem gives the lower complexity bounds for affinely constrained problems in the form of \eqref{eq:ECO}.
\begin{theorem}[lower complexity bounds for affinely constrained problems]
	Let $m\le n$ and $t < \frac{m}{4}-2$ be positive integers, and $L_f>0$. For any first-order method $\cM$ that is described in \eqref{eq:I}, there exists a problem instance in the form of \eqref{eq:ECO} such that $\nabla f$ is $L_f$-Lipschitz continuous,  the instance has a primal-dual solution $(\vx^*,\vy^*)$, and 
		\begin{align*}
		\big|f( \xu^{(t)} )-f(\vx^*)\big|\ge & ~ \frac{3L_f\|\vx^*\|^2}{128(2t+5)^2} + \frac{\sqrt{3}\|\vA\|\cdot\|\vx^*\|\cdot\|\vy^*\|}{8(2t+5)}, 
		\\[0.1cm]
		\| \vA  \xu^{(t)} - \vb \| \ge &~  \frac{\sqrt{3}\|\vA\|\cdot\|\vx^*\|}{4\sqrt{2}(2t+5)}.
		\end{align*}
	where $ \bar\vx^{(t)} $ is the approximate solution output by $\cM$. In addition, given $\mu>0$, there exists an instance of \eqref{eq:ECO} such that $f$ is $\mu$-strongly convex, the instance has a primal-dual solution $(\vx^*,\vy^*)$, and 
	\begin{equation*}
	\|\xu^{(t)}-\vx^*\|^2 \ge \frac{5\|\vA\|^2\cdot \|\vy^*\|^2}{256\mu^2 (2t+5)^2},
	\end{equation*}
\end{theorem}
For a general convex-concave bilinear saddle-point problem \eqref{eq:saddle-prob}, we obtain lower complexity bound results as summarized by the theorem below.
\begin{theorem}[lower complexity bounds for bilinear saddle-point problems]
	Let $m\le n$ and $t < \frac{m}{4}-2$ be positive integers, and $L_f>0$. For any first-order method $\cM$ that is described in \eqref{eq:I}, there exists a problem instance in the form of \eqref{eq:saddle-prob} such that $\nabla f$ is $L_f$-Lipschitz continuous, $X$ and $Y$ are Euclidean balls with radii $R_X$ and $R_Y$ respectively, and
	\begin{align*}
	\phi( \xu^{(t)} )-\psi(\yu^{(t)})\ge & ~ \frac{L_fR_X^2}{16(4t+9)^2} + \left(\frac{\sqrt{2}+2}{4}\right)\frac{\|\vA\|R_XR_Y}{4t+9},
	\end{align*}
	where $\phi$ and $\psi$ are the associated primal and dual objective functions in \eqref{eq:SPP} and \eqref{eq:SPP-dual}, and $(\xu^{(t)},\yu^{(t)})$ is the approximate solution output by $\cM$. In addition, given $\mu>0$, there exists an instance of \eqref{eq:saddle-prob} such that $f$ is $\mu$-strongly convex, $X$ and $Y$ are Euclidean balls with radii $R_X$ and $R_Y$ respectively, and
	$$\phi( \xu^{(t)} )-\psi(\yu^{(t)})\ge\frac{5\|\vA\|^2 R_Y^2}{512\mu(4t+9)^2}.$$
\end{theorem}

Comparing to upper complexity bounds of several existing first-order methods, we find that our lower complexity bounds are tight, up to the difference of constant multiples or logarithmic term.

\subsection{Literature review}

Among existing works on complexity analysis of numerical methods, many more are about showing upper complexity bounds instead of lower bounds. Usually, the upper complexity bounds are established on solving problems with specific structures. They are important because they can tell the users at most how many efforts would guarantee a desired solution. On the contrary, lower complexity bounds, which were first studied in the seminal work \cite{nemirovski1983problem}, are usually information-based and shown on solving a general class of problems. Their importance lies in telling if a certain numerical method can still be improved for a general purpose and also in guiding the algorithm designers to make ``optimal'' methods. Although there are not many works along this line, each of them sets a base for designing numerical approaches. Below we review these lower complexity bound results on different classes of problems.

\vspace{0.1cm}
\noindent\textbf{Proximal gradient methods.} On solving convex problems in the form of $F^*:=\min_\vx \{F(\vx):=f(\vx)+g(\vx)\}$, the proximal gradient method (PGM) iteratively updates the estimated solution by acquiring information of $\nabla f$ and $\prox_{\eta g}$ at certain points, where the proximal mapping of a function $h$ is defined as
$$\prox_h(\vz)=\argmin_\vx h(\vx)+\frac{1}{2}\|\vx-\vz\|^2.$$ 
For the class of problems that have $L_f$-Lipschitz continuous $\nabla f$, the lower bound has been established in \cite{nemirovski1983problem,nemirovski1992information,nesterov2004introductory,guzman2015lower}. For example,  \cite[Theorem 2.1.7]{nesterov2004introductory} establishes a lower convergence rate bound: $F(\xu^{(t)})-F^* \ge \frac{3L_f\|\vx^{(0)}-\vx^*\|^2}{32(t+1)^2}$, where $\xu^{(t)}$ is the approximate solution output by PGM after $t$ iterations, and $\vx^*$ is one optimal solution. In addition, setting $\eta=\frac{1}{L_f}$, \cite{FISTA2009, nesterov2013gradient} show that the PGM can achieve $O(L_f/t^2)$ convergence rate, and more precisely, $F(\xu^{(t)})-F^* \le \frac{2L_f\|\vx^{(0)}-\vx^*\|^2}{(t+1)^2}$. Comparing the lower and upper bounds, one can easily see that they differ only by a constant multiple. Hence, the lower bound is tight in terms of the dependence on $t$, $L_f$, and $\|\vx^{(0)}-\vx^*\|$, and also the method given in \cite{FISTA2009, nesterov2013gradient} is optimal among all methods that only access the information of $\nabla f$ and $\prox_{\eta g}$. 

For the class of problems where $f$ has $L_f$-Lipschitz continuous gradient and also is $\mu$ strongly convex, the lower bound has been established in \cite{nemirovski1983problem,nemirovsky1991optimality,nemirovski1992information,nesterov2004introductory}. For example, \cite[Theorem 2.1.13]{nesterov2004introductory} establishes a lower convergence rate bound: $F(\xu^{(t)})-F^*\ge \frac{\mu\|\vx^{(0)}-\vx^*\|^2}{2}\big(\frac{\sqrt\kappa-1}{\sqrt\kappa+1}\big)^{2t}$, where $\kappa=\frac{L_f}{\mu}$ denotes the condition number. In addition, assuming the knowledge of $\mu$ and $L_f$, \cite[Theorem 6]{nesterov2013gradient} shows the convergence rate: $F(\xu^{(t)})-F^*\le \frac{L_f\|\vx^{(0)}-\vx^*\|^2}{4}\big(1+\frac{1}{\sqrt{2\kappa}}\big)^{-2t}$. Note that both lower and upper bounds are linear, and they have the same dependence on $\|\vx^{(0)}-\vx^*\|$ and $\kappa$. In this sense, the lower bound is tight, and the method is optimal.

\vspace{0.1cm}
\noindent\textbf{Inexact gradient methods.} On the convex problem $f^*:=\min_{\vx} f(\vx)$ for which only inexact approximation of $\nabla f$ is available, there have been several studies on the corresponding lower complexity bound. For example, on solving the convex stochastic program $f^*:=\min_{\vx} \{f(\vx):=\EE_\xi f_\xi(\vx)\}$, the stochastic gradient method (SGM) performs iterative update to the solution by accessing the stochastic approximation of subgradient $\tilde\nabla f$ at a certain point. For the class of problems whose $f$ is Lipschitz continuous, \cite{nemirovski1983problem, nesterov2004introductory} show that to find an $\vareps$-optimal solution $\bar\vx$, i.e., $f(\bar\vx)-f^* \le \vareps$, the algorithm needs to run $O(1/\vareps^2)$ iterations. On the other hand, as shown in \cite{nemirovski2009robust}, the order $1/\vareps^2$ is achievable with appropriate setting of algorithm parameters. Hence, the lower complexity bound $O(1/\vareps^2)$ is tight, and the stochastic gradient method is optimal on finding an approximate solution to the convex stochastic program. Further study of lower complexity bound of inexact gradient methods is also performed in \cite{devolder2014first}. When $f(\vx)$ has a special finite-sum structure, the lower complexity bound of randomized gradient method is studied in \cite{lan2017optimal,woodworth2016tight,arjevani2016dimension}. 

\vspace{0.1cm}
\noindent\textbf{Primal-dual first-order methods.} On an affinely constrained problem \eqref{eq:ECO} or the more general saddle-point problem \eqref{eq:saddle-prob}, many works have studied primal-dual first-order methods, e.g., \cite{chambolle2011first,esser2010general,condat2013primal, chen2014optimal, ouyang2015accelerated, lan2016iteration-alm, he2016accelerated, GXZ-RPDCU2016, gao2017first, xu2017iter-complexity-ialm, yan2018new-pds, hamedani2018primal}. To obtain an $\vareps$-optimal solution in a certain measure, an $O(1/\vareps)$ complexity result is established by many of them for convex problems. In addition, for strongly convex cases, an improved result of $O(1/\sqrt\vareps)$ has been shown in a few works such as \cite{goldstein2014fast,xu2017iter-complexity-ialm, xu2018accelerated-pdc,  hamedani2018primal}. All these results are about upper complexity bounds and none about lower bounds. Hence, it is unclear if these methods achieve the optimal order of convergence. Our results will fill the missing part and can be used to determine the optimality of these existing algorithms.

\vspace{.1cm}
\noindent\textbf{Others.} In adddition to the above list of lower complexity bounds, there are also a few results on special types of problems. The lower complexity bound of subgradient methods for uniformly convex optimization has been studied in \cite{juditsky2014deterministic}. Under the assumption that an algorithm has access to gradient information and is only allowed to perform linear optimization (instead of computing a projection), the lower complexity bounds have been studied in \cite{jaggi2013revisiting,lan2013complexity}. The lower complexity bounds of oblivious algorithms are studied in \cite{arjevani2016iteration}, where the way to generate new iterates by the algorithms are restricted. The lower complexity bound of stochastic gradient algorithms that preserves local privacy is studied in \cite{duchi2012privacy}.

\subsection{Notation and outline}
We use bold lower-case letters $\vx,\vy,\vc, \ldots$ for vectors and bold upper-case letters $\vA, \vQ,\ldots$ for matrices. For any vector $\vx\in\R^n$, we use $x_i$ to denote its $i$-th component. When describing an algorithm, we use $\vx^{(k)}$ for the $k$-th iterate. $\vA^\top$ denotes the transpose of a matrix $\vA$. We use $\vzero$ for all-zero vector and $\vone$ for all-one vector, and we use $\vO$ for an zero matrix and $\vI$ for the identity matrix. Their sizes will be specified by a subscript, if necessary, and otherwise are clear from the context. $\ve_{j,p}=[0,\ldots,0,1,0,\ldots,0]^\top\in\RR^p$ denotes the $j$-th standard basis vector in $\RR^p$. We use $\ZZ_{++}$ for the set of positive integers and $\SS_+^n$ for the set of all $n\times n$ symmetric positive semidefinite matrices. Without further specification, $\|\cdot\|$ is used for the Euclidean norm of a vector and the spectral norm of a matrix.

The rest of the paper is organized as follows. In section \ref{sec:linear_span_case}, for affinely constrained problems, we present lower complexity bounds of first-order methods that satisfy a linear span requirement. We drop the linear span assumption in section \ref{sec:ECO_general_case} and show lower complexity bounds of first-order methods that are described in \eqref{eq:I}. Section \ref{sec:lb-spp} is about the bilinear saddle-point problems. Lower complexity bounds are established there for first-order methods described in \eqref{eq:I}. In section \ref{sec:tightness}, we show the tightness of the established lower complexity bounds by comparing them with existing upper complexity bounds. Finally, section \ref{sec:conclusion} proposes a few interesting topics for future work and concludes the paper. 

\section{Lower complexity bounds under linear span assumption for affinely constrained problems}
\label{sec:linear_span_case}

In this and the next sections, we study lower complexity bounds of first-order methods on solving the affinely constrained problem \eqref{eq:ECO}. Our approach is to design a ``hard'' problem instance such that the convergence speed of any first-order method is lower bounded. The designed instances are convex quadratic programs in the form of
\begin{align}
\label{eq:ECO_H}
\begin{aligned}
f^*:=\min_{\vx\in\R^n}&\,\Set{ f(\vx):=\frac{1}{2}\vx^\top\vH\vx - \vh^\top\vx}
\\
\st &\ \vA \vx = \vb,
\end{aligned}
\end{align}
where $\vA\in\RR^{m\times n}$, and $\vH\in\SS_+^n$. Note that $\nabla f$ is Lipschitz continuous, and thus the above problem is a special case of \eqref{eq:ECO}. 

Throughout this section, we assume that the dimensions $m, n\in\ZZ_{++}$ are given and satisfy $m\le n$, and that a fixed positive integer number $k < \frac{m}{2}$ is specified. Our lower complexity analysis will be based on the performance of the $k$-th iterate of a first-order method on solving the designed instance. It should be noted that the assumption $k< \frac{m}{2}$ is valid if the problem dimensions $m$ and $n$ are very big and we do not run too many iterations of the algorithm. 

To have a relatively simple start, we focus on a special class of first-order methods in this section. More precisely, we make the following assumption.
\begin{assum}[linear span]
	\label{assum:linear_span}
The iterate sequence $\{\vx^{(t)}\}_{t=0}^\infty$ satisfies $\vx^{(0)}=\vzero$ and
	\begin{equation*}
\vx^{(t)} \in \Span\left\{ \nabla f( \vx^{(0)}),  \vA ^\top  \vr^{(0)}, \nabla f( \vx^{(1)}),  \vA ^\top  \vr^{(1)}, \ldots, \nabla f( \vx^{(t-1)}),  \vA ^\top  \vr ^{(t-1)}\right\},\, t\ge 1,
\end{equation*}
where $\vr=\vA\vx-\vb$ denotes the residual.
\end{assum}
In the context, we refer to the above assumption as the linear span assumption. It is easy to see that if $X=\RR^n$, then the LALM with updates in \eqref{eq:update-lalm} satisfies this assumption. In addition, it is not difficult to find rules $\{\cI_t\}_{t=0}^\infty$ such that the iterate sequence $\{\vx^{(t)}\}$ in Assumption \ref{assum:linear_span} can be obtained by \eqref{eq:I}. Note that we do not lose generality by assuming $\vx^{(0)}=\vzero$, because otherwise we can consider a shifted problem
$$\min_{\vx} f(\vx-\vx^{(0)}), \st \vA(\vx-\vx^{(0)})=\vb.$$ 

It should be noted that Assumption \ref{assum:linear_span} may not always hold, e.g., when there is projection involved in the description of a first-order algorithm.  
The lower complexity bound analysis can be performed without the linear span assumption, thanks to a technique introduced in \cite{nemirovski1992information,nemirovsky1991optimality} that utilize a certain rotational invariance of quadratic functions over a Euclidean ball. To facilitate reading, we defer the incorporation of such a technique to section \ref{sec:ECO_general_case}, 
where we will elaborate on the technical details and perform the lower complexity bound analysis without Assumption \ref{assum:linear_span}.

\subsection{Special linear constraints}
\label{sec:special-inst}

In this subsection, we describe a set of special linear constraints, which will be used to study the lower complexity bound of first-order methods satisfying Assumption \ref{assum:linear_span}. 

We let the matrix $\vLam$ and vector $\vc$ be
\begin{align}
	\label{eq:Lamc}
	\vLam= \left[\begin{array}{cc} \vB  &  \vO  \\  \vO  &  \vG \end{array}\right]\in\RR^{m\times n}
	\text{ and }
	\vc  =  \left[\begin{array}{c} \vone_{2k} \\  \vzero  \end{array}\right]\in\RR^m,	
\end{align}
where $\vG\in\R^{(m-2k)\times(n-2k)}$ is any matrix of full row rank such that $\|\vG\| = 2$, and 
\begin{align}
	\label{eq:B}
	\vB :=\left[\begin{array}{rrrrr}
	&  &  &-1 & \ 1\\
	& & \iddots & \iddots &\\
	& -1 & 1 &  & \\ 
	-1 & 1 &  & & \\
	1 &  &  & & 
	\end{array}\right]\in\RR^{2k\times 2k}.
\end{align}
All the designed ``hard'' instances in this paper are built upon $\vLam$ and $\vc$ given in \eqref{eq:Lamc}. 
Two immediate observations regarding \eqref{eq:Lamc} and \eqref{eq:B} are as follows. 
First,
for any $\vu:=(u_1,\ldots,u_{2k})^\top\in\R^{2k}$, we have
\begin{align*}
\|\vB\vu\|^2 = (u_{2k} - u_{2k-1})^2 + \cdots + (u_2 - u_1)^2 + u_1^2 \le 2(u_{2k}^2 + u_{2k-1}^2) + \cdots + 2(u_2^2 + u_1^2) + u_1^2\le 4\|\vu\|^2,
\end{align*}
so 
\begin{align}
	\label{eq:Bnorm}
	\|\vB\|\le 2.
\end{align}
Consequently, noting $\|\vG\|=2$ and the block diagonal structure of $\vLam$, we have
\begin{align}
	\label{eq:norm_Lam}
	\|\vLam\| = \max\{\|\vB\|,\|\vG\|\} = 2.
\end{align}
Second, 
it is straightforward to verify that 
\begin{align}
	\label{eq:invB}
	\vB^{-1}=\left[\begin{array}{cccc}& & & 1\\ & &  1 & 1\\ & \iddots & \vdots & \vdots\\ 1 & \cdots & 1 & 1 \end{array}\right].
\end{align}
Based on the two observations, we have the following lemma. 
\begin{lemma}
	\label{lem:Lamc2Ab}
	Let $\vLam$ and $\vc$ be given in \eqref{eq:Lamc}, and for any $L_A>0$, let 
	\begin{align}
		\label{eq:Ab}
		\vA = \frac{L_A}{2}\vLam\text{ and }\vb = \frac{L_A}{2}\vc.
	\end{align}
	Then $\|\vA\| = L_A$. In addition, for any vector $\vx^*=(x_1^*,\ldots,x_n^*)^\top$ that satisfies $\vA\vx^* = \vb$ we have
		\begin{equation*}
			x_i^*= i, \text{ for } 1\le i\le 2k.
		\end{equation*}
\end{lemma}
	\begin{proof}
		By \eqref{eq:norm_Lam} and the definition of $\vA$ in \eqref{eq:Ab}, we immediately have $\|\vA\|=L_A$. To solve the linear system $\vA\vx = \vb$, we split $\vx$ into two parts as $\vx=(\vu^\top, \vv^\top)^\top$ with $\vu\in\RR^{2k}$ and $\vv\in\RR^{n-2k}$. Then by the block diagonal structure of $\vA$, we have from $\vA\vx = \vb$ and the definitions of $\vA$ and $\vb$ in \eqref{eq:Ab} that $\vB\vu = \vone_{2k}$.
		It follows from \eqref{eq:invB} that $\vB$ is nonsingular, and thus the linear system $\vB\vu = \vone_{2k}$ has a unique solution 
		\begin{align}
			\label{eq:ustar}
			\vu^*=\vB^{-1}\vone_{2k}=(1,\ldots,2k)^\top, 
		\end{align}
which completes the proof.
	\end{proof}

\subsection{Krylov subspaces}
In this subsection, we study two Krylov subspaces that are associated with the matrix $\vLam$ and vector $\vc$ described in \eqref{eq:Lamc}. 
In particular, we consider the Krylov subspaces
\begin{align}
\label{eq:KJ}
\cJ_i:=\Span\{\vc, (\vLam \vLam^\top) \vc,(\vLam \vLam^\top)^2\vc,\ldots,(\vLam \vLam^\top)^i\vc\}\subseteq\R^{m} \text{ and } \cK_i := \vLam^\top \cJ_i\subseteq\R^{n},\,\text{ for }i\ge 0.
\end{align}
As shown below in \eqref{eq:power-AtAb}, restricting on the first $2k$ entries, the above two Krylov subspaces reduce to
\begin{align}
		\label{eq:RFdef}
		\cF_i:=\Span\{\vone_{2k},\vB^2 \vone_{2k},\ldots,\vB^{2i}\vone_{2k}\}\text{ and }\cR_i:=\Span\{\vB \vone_{2k},\ldots,\vB^{2i+1}\vone_{2k}\}.
		\end{align}
We first establish some important properties of $\cF_i$ and $\cR_i$ as follows.	
\begin{lemma}
	\label{lem:krylov-FR}
Let $\cF_i$ and $\cR_i$ be defined in \eqref{eq:RFdef}. For any $0\le i\le 2k-1$, we have
	\begin{align}
		\label{eq:RFclaim}
		\cF_i = \Span\{\vone_{2k},\ve_{1,2k},\ve_{2,2k},\ldots,\ve_{i,2k}\},\quad \cR_i=\Span\{\ve_{2k-i,2k},\ve_{2k-i+1,2k},\ldots,\ve_{2k,2k}\},
		\end{align}
		and 
		\begin{align}
			\label{eq:BR}
			\vB\cR_i = \Span\{\ve_{1,2k},\ve_{2,2k},\ldots,\ve_{i+1,2k}\}\subseteq\cF_{i+1},
		\end{align}
where we have used the convention $\ve_{0,2k}=\vzero$.		
	\end{lemma}
	
\begin{proof}
From the definition of $\vB$ in \eqref{eq:B}, we have 
		\begin{align}
		\label{eq:Bprop}
		\begin{aligned}
		&\vB\vone_{2k} = \ve_{2k,2k}, \vB\ve_{2k,2k}=\ve_{1,2k},\\
		 &\vB\ve_{i,2k} = \ve_{2k-i+1,2k} - \ve_{2k-i,2k},\ \forall i=1,\ldots,2k-1.
		 \end{aligned}
		\end{align}
		Hence, from \eqref{eq:RFdef} and \eqref{eq:Bprop}, it holds that
		\begin{align*}
		\cF_0 = & \Span\{\vone_{2k}\},
\\
		\cR_0 = & \Span\{\vB\vone_{2k}\} = \Span\{\ve_{2k,2k}\}, 
		\\
		\vB\cR_0 = & \Span\{\vB\ve_{2k,2k}\} = \Span\{\ve_{1,2k}\},
		\\
\cF_1 = & \Span\{\vone_{2k}, \vB^2\vone_{2k}\} = \Span\{\vone_{2k}, \ve_{1,2k}\},
		\\
		\cR_1 = & \Span\{\vB\vone_{2k}, \vB^3\vone_{2k}\} = \Span\{\ve_{2k,2k}, \vB\ve_{1,2k}\} = \Span\{\ve_{2k-1,2k}, \ve_{2k,2k}\},
		\\
		\vB\cR_1 = & \Span\{\vB^2\vone_{2k}, \vB^4\vone_{2k}\} = \Span\{\ve_{1,2k}, \vB^2\ve_{1,2k}\} = \Span\{\ve_{1,2k},\ve_{2,2k}\}.
		\end{align*}
		Therefore, the results in \eqref{eq:RFclaim} and \eqref{eq:BR} hold for $i=0$ and $i=1$.
		
		Below we prove the results by induction. Assume that there is a positive integer $s < 2k$ and \eqref{eq:RFclaim} holds for $i = s-1$, namely,
		\begin{align}		
		\cF_{s-1}  = \Span\{\vone_{2k},\ve_{1,2k},\ve_{2,2k},\ldots,\ve_{s-1,2k}\},
		\ \cR_{s-1}=\Span\{\ve_{2k-s+1,2k},\ldots,\ve_{2k,2k}\}.\label{eq:RFhypo}
		\end{align}
		From \eqref{eq:Bprop} and \eqref{eq:RFhypo}, it follows that
		\begin{equation}\label{eq:vbspan}
		\vB\cR_{s-1}=\vB\Span\{\ve_{2k-s+1,2k},\ldots,\ve_{2k,2k}\} \subseteq \Span\{\ve_{s,2k},\ve_{s-1,2k},\ldots,\ve_{1,2k}\}.
\end{equation}
Since $\vB$ is nonsingular, 
$\dim\big(\vB\cR_{s-1}\big)=\dim\big(\cR_{s-1}\big)=s.$
Hence, from \eqref{eq:vbspan} and also noting $$\dim\big(\Span\{\ve_{s,2k},\ve_{s-1,2k},\ldots,\ve_{1,2k}\}\big)=s,$$ we have
$$\vB\cR_{s-1} = \Span\{\ve_{s,2k},\ve_{s-1,2k},\ldots,\ve_{1,2k}\}.$$

Observing $\Span\{\vB^2 \vone_{2k},\ldots,\vB^{2s}\vone_{2k}\}=\vB \cR_{s-1}$	, we conclude
$$\cF_s=\Span\{\vone_{2k}, \vB^2 \vone_{2k},\ldots,\vB^{2s}\vone_{2k}\} =  \Span\{\vone_{2k},\ve_{1,2k},\ve_{2,2k},\ldots,\ve_{s,2k}\}.$$
Through essentially the same arguments, one can use \eqref{eq:Bprop}, the above equation, and the fact $\cR_s=\vB\cF_s$ to conclude
$$\cR_s=\Span\{\ve_{2k-s,2k},\ldots,\ve_{2k,2k}\},$$
and thus we complete the proof.	
\end{proof}	
	
Through relating $\cJ_i$ (resp. $\cK_i$) to $\cF_i$ (resp. $\cR_i$), we have the following result.
\begin{lemma}
	\label{lem:krylov-JK}
Let $\cJ_i$ and $\cK_i$ be defined in \eqref{eq:KJ}.	For any $0\le i\le 2k-1$, it holds
	\begin{align}
	\label{eq:krylov}
	\cJ_i=\Span\{\vc,\ve_{1,m},\ve_{2,m},\ldots,\ve_{i,m}\},\ \cK_i = \Span\{\ve_{2k-i,n},\ve_{2k-i+1,n},\ldots,\ve_{2k,n}\},
	\end{align}
	and
	\begin{align}
		\label{eq:LamK}
		\vLam\cK_i = \Span\{\ve_{1,m},\ve_{2,m},\ldots,\ve_{i+1,m}\}\subseteq\cJ_{i+1}.
	\end{align}
	\end{lemma}
	
	\begin{proof}
		Observe that for any $i=0,\ldots,2k-1$ we have
		\begin{align}\label{eq:power-AtAb}
		(\vLam\vLam^\top)^i \vc = \begin{pmatrix}
		\vB^{2i}\vone_{2k}\\ \vzero_{m-2k}
		\end{pmatrix}
		\text{ and }
		\vLam^\top (\vLam\vLam^\top)^i \vc = \begin{pmatrix}
		\vB^{2i+1}\vone_{2k}\\ \vzero_{n-2k}
		\end{pmatrix}.
		\end{align}
		Consequently, the definitions in \eqref{eq:KJ} becomes 
		\begin{align*}
		\cJ_i = \cF_i\times\{\vzero_{m-2k}\}\text{ and }\cK_i = \cR_i\times\{\vzero_{n-2k}\}.
		\end{align*}
		Therefore, the results in \eqref{eq:krylov} and \eqref{eq:LamK} immediately follow from Lemma \ref{lem:krylov-FR}. 		
	\end{proof}

Two remarks are in place for the Krylov subspaces $\cK_i$ and $\cJ_i$. 
First, 
by the definitions of $\cK_i$ and $\cJ_i$ in \eqref{eq:KJ} and the relation \eqref{eq:LamK}, we have 
\begin{align}
\label{eq:AonKJ}
\vLam\cK_i \subseteq \cJ_{i+1}\text{ and }\vLam^\top\cJ_i = \cK_i.
\end{align}
Second, 
by \eqref{eq:krylov} we have
\begin{align}
\label{eq:KJexpan}
\cK_{i-1}\subsetneq \cK_{i}\text{ and }\cJ_{i-1}\subsetneq\cJ_{i},\ \forall i=1,\ldots,2k-1.
\end{align}

\subsection{A lower complexity bound with positive $L_A$}
In this subsection, we establish a lower complexity bound of any first-order method that satisfies Assumption \ref{assum:linear_span} on solving \eqref{eq:ECO_H}. Our approach is to build an instance such that the iterate $\vx^{(t)}\in\cK_{t-1},\,\forall t\le k$ and then estimate the values
\begin{align}
	\label{eq:measures}
	\min_{\vx^{(t)}\in\cK_{t-1}}|f(\vx^{(t)}) - f^*|,\text{ and }\min_{\vx^{(t)}\in\cK_{t-1}}\|\vA\vx^{(t)} - \vb\|.
\end{align}
In the above equation, the former value is used to measure the performance of an algorithm by the objective value difference and the latter by the feasibility error. It should be noted that the absolute value is needed in the former measure, since it is possible that $f(\vx^{(t)})< f^*$ when $\vx^{(t)}$ is not a feasible point. 

The following lemma specifies the conditions on \eqref{eq:ECO_H} to guarantee $\vx^{(t)}\in\cK_{t-1},\,\forall t\le k$.
\begin{lemma}\label{lem:xt-sp}
	Given $L_A\in\RR$, let $\vA$ and $\vb$ be those in \eqref{eq:Ab}. Consider \eqref{eq:ECO_H} with $\vh\in\cK_{0}$ and $\vH$ satisfying $\vH\cK_{t-1}\subseteq\cK_t$ for any $1\le t\le k$, where $\cK_{i}$ is defined in \eqref{eq:KJ}. 
	Then under Assumption \ref{assum:linear_span}, we have $\vx^{(t)}\in \cK_{t-1}$ for any $1\le t \le k$. 
\end{lemma}

\begin{proof}
It suffices to prove that for any $t=1,\ldots,k$,
	\begin{align}
	\label{eq:span2K}
	\Span\left\{ \nabla f( \vx^{(0)}),  \vA ^\top  \vr^{(0)}, \nabla f( \vx^{(1)}),  \vA ^\top  \vr^{(1)}, \ldots, \nabla f( \vx^{(t-1)}),  \vA ^\top  \vr ^{(t-1)}\right\}\subseteq \cK_{t-1}.
	\end{align}
	We prove the result by induction. First, since $ \vx^{(0)}= \vzero $, from \eqref{eq:Lamc} and \eqref{eq:Ab} we have $ \vA ^\top  \vr^{(0)} = -\vA^\top \vb \in\Span\{\ve_{2k,n}\} =\cK_0$. In addition, from the condition $\vh\in\cK_0$, it follows that $\nabla f(\vx^{(0)}) = -\vh \in\cK_0$. Therefore, \eqref{eq:span2K} holds when $t=1$. Assume that for a certain $1\le s<k$, \eqref{eq:span2K} holds for $t=s$, 
	and consequently
	\begin{equation}\label{eq:assum-ind}
	\vx^{(s)}\in\cK_{s-1}.
	\end{equation} 
	We go to prove the result in \eqref{eq:span2K} for $t=s+1$, or equivalently $\nabla f(\vx^{(s)}),\vA^\top\vr^{(s)}\in\cK_s$, and finish the induction. From \eqref{eq:KJexpan} we have $\cK_0\subseteq\cK_s$. By this observation, noting $\vx^{(s)}\in \cK_{s-1}$, and using the conditions $\vh\in\cK_0$ and $\vH\cK_{s-1}\subseteq\cK_s$, we have $\nabla f(\vx^{(s)}) = \vH\vx^{(s)} - \vh\in\cK_s$.
	In addition, from \eqref{eq:AonKJ} and \eqref{eq:assum-ind}, we have $\vA^\top \vA \vx^{(s)} \in \cK_{s}$. Since $\vA^\top \vb\in \cK_0\subseteq \cK_s$, then $\vA^\top \vr^{(s)}= \vA^\top \vA \vx^{(s)} - \vA^\top \vb \in \cK_{s}$. Therefore, $\nabla f(\vx^{(s)})$ and $\vA^\top\vr^{(s)}$ are both in $\cK_s$, and by induction \eqref{eq:span2K} holds for any $1\le t\le k$. This completes the proof.
	\end{proof}

Based on the above lemma, we construct an instance of \eqref{eq:ECO_H} that satisfies the conditions $\vh\in\cK_{0}$ and $\vH\cK_{t-1}\subseteq\cK_t$ for any $1\le t\le k$. 
Given positive numbers $L_f$ and $L_A$, we build a quadratic program as 
	\begin{align}
		\label{eq:ECO_Hh0}
		\begin{aligned}
		f^*:=\min_{\vx\in\R^n}&\,\Set{ f(\vx) := L_f\left(\frac{1}{2}x_k^2 + \frac{1}{2}\sum_{i=2k+1}^{n}x_i^2\right)}
		\\
		\st &\ \vA\vx=\vb,
		\end{aligned}
\end{align}
where $\vA$ and $\vb$ are those in \eqref{eq:Ab}. Clearly, \eqref{eq:ECO_Hh0} is a special instance of \eqref{eq:ECO_H} with $\vh=\vzero$ and a diagonal $\vH\in\SS_+^n$. It is obvious to see $\vh\in\cK_0$, and since $\vH$ is diagonal, it is easy to verify $\vH\cK_{t-1}\subseteq\cK_t$ for any $1\le t\le k$. Hence from Lemma \ref{lem:xt-sp}, we have the following results.
\begin{lemma}\label{lem:xt-sp-Hh0}
Applying to \eqref{eq:ECO_Hh0} a first-order method that satisfies Assumption \ref{assum:linear_span}, we have $\vx^{(t)}\in \cK_{t-1}$ for any $1\le t \le k$. In addition, $f(\vx)=0$ and $\nabla f(\vx)=\vzero$ for all $\vx\in\cK_{k-1}$. 
\end{lemma}

\begin{proof}
The first result directly follows from Lemma \ref{lem:xt-sp}, and the second one can be easily verified.
\end{proof}

The lemma below characterizes the primal-dual solution and the optimal objective value $f^*$ of \eqref{eq:ECO_Hh0}.

\begin{lemma}[primal-dual solution of \eqref{eq:ECO_Hh0}]
	\label{lem:ECO_Hh0_solutions}
Let $L_f$ and $L_A$ be positive numbers. The problem instance \eqref{eq:ECO_Hh0} has a unique optimal solution $\vx^*$ with a unique associated Lagrange multiplier $\vy^*$ given by
	\begin{equation}
	\label{eq:sol-lam=0}
	x_i^*=\left\{\begin{array}{ll} i, &\text{ if } 1\le i\le 2k,\\[0.1cm]
	0, & \text{ if } i\ge 2k+1, \end{array}\right.
	\end{equation}
	and
	\begin{equation}
	\label{eq:sol-lam=0-y}
	\vy^*= \frac{2kL_f}{L_A} \left[\begin{array}{c}\vzero_k \\ \vone_k \\ \vzero\end{array}\right],
	\end{equation}
	respectively. In addition, the optimal objective value is $f^* = \frac{L_f}{2}k^2$.
\end{lemma}
\begin{proof}
	Similar to the proof of Lemma \ref{lem:Lamc2Ab}, we split $\vx$ into two parts as $\vx=(\vu^\top, \vv^\top)^\top$ with $\vu\in\RR^{2k}$ and $\vv\in\RR^{n-2k}$. Then from the block structure of $\vA$, it follows that \eqref{eq:ECO_Hh0} is equivalent to the following two smaller problems:
	\begin{align}
	\label{eq:1st-subprob}
	&\min_\vu \frac{L_f}{2}u_k^2, \st \frac{L_A}{2}\vB\vu=\frac{L_A}{2}\vone_{2k},
	\\[0.1cm]
	\label{eq:2st-subprob}
	&\min_\vv \frac{L_f}{2}\|\vv\|^2, \st \frac{L_A}{2}\vG\vv=\vzero.
	\end{align} 
	By Lemma \ref{lem:Lamc2Ab}, the former problem \eqref{eq:1st-subprob} has a unique feasible (and thus optimal) solution $\vu^*$ that is given in \eqref{eq:ustar}. Clearly, the latter problem \eqref{eq:2st-subprob} has a unique solution  $\vv^*=\vzero$.
	Hence we obtain \eqref{eq:sol-lam=0}. Consequently,
	\begin{align*}
	f^* = \frac{L_f}{2}\left(u_k^*\right)^2 = \frac{L_f}{2}k^2.
	\end{align*}
	
	To derive the corresponding Lagrange multiplier, we split the dual variable to $\vy=(\vlam^\top,\vpi^\top)^\top$ 
	with $\vlam\in\RR^{2k}$ and $\vpi\in\RR^{m-2k}$. It follows from the Karush-Kuhn-Tucker (KKT) optimality conditions of \eqref{eq:1st-subprob} and the solution $\vx^*$ in \eqref{eq:sol-lam=0} that
	\begin{equation*}
	\frac{L_A}{2}\vB^\top \vlam^* = L_fu_k^*\ve_{k,2k} = L_fk\ve_{k,2k}, \quad \vG^\top\vpi^* = \vzero.
	\end{equation*}
	Since $\vG$ is full row rank, we have $\vpi^*=\vzero$. 
	In addition, from \eqref{eq:invB} we have
	\begin{align*}
	\vlam^* = \frac{2}{L_A}\left(\vB^\top\right)^{-1}(L_fk\ve_{k,2k}) = \frac{2L_f}{L_A}k\begin{bmatrix}
	\vzero_k\\\vone_k
	\end{bmatrix},
	\end{align*}
	and \eqref{eq:sol-lam=0-y} follows immediately. 
\end{proof}

By Lemmas \ref{lem:xt-sp-Hh0} and \ref{lem:ECO_Hh0_solutions}, we can easily estimate the values in \eqref{eq:measures} as follows.
\begin{lemma}
	\label{lem:ECO_Hh0_KkBounds}
Let $L_f$ and $L_A$ be positive numbers.	For the problem instance \eqref{eq:ECO_Hh0}, we have
	\begin{subequations}
		\label{eq:ECO_Hh0_KkBounds}
		\begin{align}
		\label{eq:ECO_Hh0_KkBounds_obj}
		\min_{\vx\in\cK_{k-1}}\big|f(\vx) - f^*\big|\ge & ~ \frac{3L_f\|\vx^*\|^2}{32(k+1)} + \frac{\sqrt{6}}{32(k+1)}L_A\|\vx^*\|\cdot\|\vy^*\|, 
		\\[0.1cm]
		\label{eq:ECO_Hh0_KkBounds_feas}
		\min_{\vx\in\cK_{k-1}}\|\vA\vx - \vb\| \ge &~  \frac{\sqrt{3}L_A\|\vx^*\|}{4\sqrt{2}(k+1)},
		\end{align}
	\end{subequations}
	where $(\vx^*,\vy^*)$ is the unique primal-dual solution pair of \eqref{eq:ECO_Hh0}, and $\cK_{k-1}$ is defined in \eqref{eq:KJ}.
\end{lemma}

\begin{proof}
	Using the formula 
	\begin{equation}\label{eq:sum-i2}
	\sum_{i=1}^p i^2 = \frac{p(p+1)(2p+1)}{6}
	\end{equation}
	and the description of $\vx^*$ in \eqref{eq:sol-lam=0}, we have 
	\begin{equation}\label{eq:ECO_Hh0_nrm-x-lam=0}
	\| \vx ^*\|^2 =\sum_{i=1}^{2k}i^2= \frac{k(2k+1)(4k+1)}{3}.
	\end{equation}
	For any $\vx\in\cK_{k-1}$, we observe from \eqref{eq:Lamc}, \eqref{eq:Ab} and \eqref{eq:LamK} that $ \vA  \vx$ can only have nonzeros on its first $k$ components. Since the first $2k$ components of $\vb$ all equal $\frac{L_A}{2}$, we have
	\begin{equation}\label{eq:ineq-Ax-b-lem-Hh0}
	\|\vA\vx-\vb\|^2\ge \frac{k L_A^2}{4} \overset{\eqref{eq:ECO_Hh0_nrm-x-lam=0}}= \frac{3L_A^2\|\vx^*\|^2}{4(2k+1)(4k+1)}\ge \frac{3L_A^2\|\vx^*\|^2}{32(k+1)^2},
	\end{equation}
	and hence \eqref{eq:ECO_Hh0_KkBounds_feas} holds.

	In addition, as $\vx\in\cK_{k-1}$, we have from Lemma \ref{lem:xt-sp-Hh0} that $f(\vx)=0$. Hence, it follows from Lemma \ref{lem:ECO_Hh0_solutions} that
	\begin{equation*}
	\left|f( \vx )-f( \vx ^*)\right| = f^* = \frac{L_f k^2}{2}.
	\end{equation*}
	For $k\ge 1$, it is easy to verify that
	\begin{equation}\label{eq:obj-bd-t1}
	\frac{L_f k^2}{2}\overset{\eqref{eq:ECO_Hh0_nrm-x-lam=0}}= \frac{3L_fk\|\vx^*\|^2}{2(2k+1)(4k+1)} 
	\ge \frac{3L_f\|\vx^*\|^2}{16(k+1)}.
	\end{equation}
	Also from \eqref{eq:sol-lam=0-y}, we have 
	\begin{align}
	\label{eq:ECO_Hh0_nrm-y-lam=0}
	\|\vy^*\|^2=\frac{4L_f^2}{L_A^2}k^3, 
	\end{align}
	and it is not difficult to verify that
	\begin{equation}\label{eq:obj-bd-t2}
	\frac{L_f k^2}{2}\ge \frac{\sqrt{6}}{16(k+1)}2 L_f k^{\frac{3}{2}}\frac{\sqrt{k(2k+1)(4k+1)}}{\sqrt{3}} = \frac{\sqrt{6}}{16(k+1)} L_A\|\vx^*\|\cdot\|\vy^*\|,
	\end{equation}
where the second equality uses \eqref{eq:ECO_Hh0_nrm-x-lam=0} and \eqref{eq:ECO_Hh0_nrm-y-lam=0}. Therefore, \eqref{eq:ECO_Hh0_KkBounds_obj} follows from \eqref{eq:obj-bd-t1} and \eqref{eq:obj-bd-t2} and the fact $\max\{a,b\}\ge\frac{a+b}{2}$, and we complete the proof.	
\end{proof}

Using Lemmas \ref{lem:xt-sp-Hh0} through \ref{lem:ECO_Hh0_KkBounds} established above, we are now ready to show our first lower complexity bound result.
\begin{theorem}[lower complexity bound I under linear span assumption]
	\label{thm:ECO_Hh0_lb-complexity-lam=0}
	Let $m\le n$ be positive integers, and $L_f$ and $L_A$ be positive numbers. For any positive integer $t<\frac{m}{2}$, there exists an instance of \eqref{eq:ECO} such that $\nabla f$ is $L_f$-Lipschitz continuous, $\|\vA\|=L_A$, and it has a primal-dual solution $(\vx^*,\vy^*)$. In addition, for any algorithm on solving \eqref{eq:ECO}, if it satisfies Assumption \ref{assum:linear_span}, then we have
	\begin{subequations}
		\label{eq:ECO_Hh0_obj-feas-err-lam=0}
		\begin{align}
			\label{eq:ECO_Hh0_obj-err-lam=0} 
			\big|f( \vx^{(t)} )-f( \vx ^*)\big|\ge & ~ \frac{3L_f\|\vx^*\|^2}{32(t+1)} + \frac{\sqrt{6}}{32(t+1)}L_A\|\vx^*\|\cdot\|\vy^*\|, 
			\\[0.1cm]
			\label{eq:ECO_Hh0_feas-err-lam=0}
			\| \vA  \vx^{(t)} - \vb \| \ge &~  \frac{\sqrt{3}L_A\|\vx^*\|}{4\sqrt{2}(t+1)}.
		\end{align}
	\end{subequations}
\end{theorem}

\begin{proof} 
Set $k=t<\frac{m}{2}$ and	consider \eqref{eq:ECO_Hh0}. Clearly, \eqref{eq:ECO_Hh0} is an instance of  \eqref{eq:ECO}, its objective $f$ has $L_f$-Lipschitz continuous gradient, and $\|\vA\|=L_A$. Its existence of a primal-dual solution is guaranteed by Lemma \ref{lem:ECO_Hh0_solutions}. By Lemma \ref{lem:xt-sp-Hh0} and also noting $t=k$, we have $\vx^{(t)}\in \cK_{k-1}$. Hence, 
	\begin{align*}
		\big|f( \vx^{(t)} )-f( \vx ^*)\big| \ge \min_{\vx\in\cK_{k-1}}\big|f(\vx) - f^*\big|,\text{ and }\| \vA  \vx^{(t)} - \vb \|\ge \min_{\vx\in\cK_{k-1}}\|\vA\vx - \vb\|.
	\end{align*}
	Now using Lemma \ref{lem:ECO_Hh0_KkBounds}, 
	we conclude \eqref{eq:ECO_Hh0_obj-err-lam=0}  and \eqref{eq:ECO_Hh0_feas-err-lam=0} immediately.	
\end{proof}

\begin{remark}\label{rmk:ECO_Hh0_lb-complexity-lam=0}
A few remarks are in place for the above theorem. 
First, 
by \eqref{eq:ECO_Hh0_obj-err-lam=0} and \eqref{eq:ECO_Hh0_feas-err-lam=0} we have $|f(\vx^{(t)}) - f^*|\ge O(1/t)$ and $\|\vA\vx^{(t)} - \vb\|\ge O(1/t)$. From the complexity point of view, given $\vareps>0$, to compute an $\vareps$-optimal solution $\vx$, i.e., $|f(\vx) - f^*|\le \varepsilon$ and $\|\vA\vx - \vb\|\le \varepsilon$, the iteration number of a first-order method is at least in the order of $1/\varepsilon$. Therefore, on solving \eqref{eq:ECO}, $O(1/\varepsilon)$ is a lower complexity bound of any first-order algorithm that satisfies Assumption \ref{assum:linear_span}.
Second, 
if we consider further the dependence of the convergence result on the norm of $\vA$, by \eqref{eq:ECO_Hh0_obj-err-lam=0} and \eqref{eq:ECO_Hh0_feas-err-lam=0} we have that the lower complexity bound is $O(L_A/\varepsilon)$. Finally, consider the dependence on the Lipschitz constant $L_f$ of the objective gradient. From \eqref{eq:ECO_Hh0_obj-err-lam=0}, the lower complexity bound is $O(L_f/\varepsilon)$ to ensure $|f(\vx) - f^*|\le \varepsilon$. It is well known that the complexity result of optimal proximal gradient methods (c.f. \cite{FISTA2009, nesterov2013gradient}) can reach the order of $\sqrt{L_f/\vareps}$, which is smaller than $L_f/\varepsilon$ if $L_f>\vareps$. We point out that the bound in \eqref{eq:ECO_Hh0_obj-err-lam=0} does not contradict to the existing results because we do not allow projection onto the linear constraint set $\vA\vx=\vb$. In addition, we point out that we do not restrict the size of primal and dual solutions in the theorem. For the designed instance \eqref{eq:ECO_Hh0}, it holds that
\begin{align}
	\label{eq:LfLECO_Hh0_relation}
	L_f\|\vx^*\|^2\approx L_A\|\vx^*\|\cdot\|\vy^*\|.
\end{align} 
Hence, the two quantities in the right hand side of \eqref{eq:ECO_Hh0_obj-err-lam=0} are of the same order, so if we change the first quantity to $\frac{3L_f\|\vx^*\|^2}{16(t+1)^2}$, we still have a valid lower bound. 

However, note that the relation in \eqref{eq:LfLECO_Hh0_relation} requires $L_A>0$, and thus our lower complexity bound results in \eqref{eq:ECO_Hh0_obj-feas-err-lam=0} do not include those for a proximal gradient method as special cases. 
In the next subsection, we drop the assumption $L_A>0$ and design another ``hard'' instance that does not have the relation \eqref{eq:LfLECO_Hh0_relation}. That instance allows us to depict a clearer picture on the dependence of the convergence results on $L_A$ and $L_f$, and the dependence coincides with the existing upper complexity bound; see \eqref{eq:t2-ouyang} in section \ref{sec:tightness}.
\end{remark}

\subsection{A lower complexity bound with nonnegative $L_A$}
\label{sec:LfLAbound}

In this subsection, we design a ``hard'' instance that is different from \eqref{eq:ECO_Hh0}. By this instance, we establish a lower complexity bound that linearly depends on $\sqrt{L_f}$ and $L_A$. This bound is particularly useful as $L_f$ is big and $L_A=O(\sqrt{L_f})$. However, to show the bound, we only need assume $L_A=O(L_f)$, which allows $L_A=0$, and thus our result can also cover the case for proximal gradient methods. 
The instance we construct is still in the form of \eqref{eq:ECO_H} with 
\begin{align}
\label{eq:HQh}
\vH = \frac{L_f}{4} \begin{bmatrix}
\vB^\top \vB &
\\
& \vI_{n-2k}
\end{bmatrix} \in\R^{n\times n}, \vh = \left(\frac{L_f}{4}+\frac{L_A}{4\sqrt{2}}\right) \ve_{2k,n}, \vA = \frac{L_A}{2}\vLam, \vb = \frac{L_A}{2}\vc,
\end{align}
where $L_f$ and $L_A$ are given nonnegative numbers, and $\vB$, $\vLam$ and $\vc$ are those given in \eqref{eq:B} and \eqref{eq:Lamc}. 
From \eqref{eq:Bnorm}, \eqref{eq:norm_Lam}, and the block diagonal structure of $\vH$ above, we have $\|\vH\| = (L_f/4)\|\vB\|^2 \le L_f$. Therefore, \eqref{eq:ECO_H} with data specified in \eqref{eq:HQh} provides an instance of \eqref{eq:ECO} whose objective gradient $\nabla f$ is $L_f$-Lipschitz continuous. 

Our derivation of the lower complexity bound follows exactly the same path as in the previous subsection.  Specifically, in the sequel we prove three lemmas that are similar to Lemmas \ref{lem:xt-sp-Hh0}, \ref{lem:ECO_Hh0_solutions}, and \ref{lem:ECO_Hh0_KkBounds}.
We show in Lemma \ref{lem:ECO_HQh_xt-sp} below that under Assumption \ref{assum:linear_span}, the iterates generated by any first-order method on solving the designed instance would satisfy $\vx^{(t)}\in\cK_{t-1}$ for any $1\le t\le k$. In Lemma \ref{lem:HQh}, we give a pair of optimal primal-dual solution and also the optimal objective value of the instance. Then, in Lemma \ref{lem:ECO_HQh_KkBounds}, we establish the lower complexity bound by estimating the values in \eqref{eq:measures}. 
\begin{lemma}
	\label{lem:ECO_HQh_xt-sp}
	Consider the instance of \eqref{eq:ECO_H} with data described in \eqref{eq:HQh}. Under Assumption \ref{assum:linear_span}, we have $\vx^{(t)}\in \cK_{t-1}$ for any $1\le t \le k$, where $\cK_{t-1}$ is defined in \eqref{eq:KJ}.
\end{lemma}

\begin{proof}
	To prove the lemma, it suffices to verify that $\vh\in\cK_{0}$ and $\vH\cK_{t-1}\subseteq\cK_t$ for any $1\le t\le k$ and then apply Lemma \ref{lem:xt-sp}. Since $\vh$ is a multiple of $\ve_{2k,n}$, from \eqref{eq:krylov} we immediately have $\vh\in\cK_0$. Using the definition of $\vH$ and the second line of equation in \eqref{eq:Bprop}, one can easily verify that $\vH\Span\{\ve_{2k-t+1,n},\ldots,\ve_{2k,n}\} = \Span\{\ve_{2k-t,n},\ve_{2k-t+1,n},\ldots,\ve_{2k,n}\}$ for any $1\le t\le k$. Hence we have all the conditions required by Lemma \ref{lem:xt-sp}, and thus $\vx^{(t)}\in \cK_{t-1}$, which completes the proof.
\end{proof}

The next lemma gives the primal-dual solution and optimal objective value of the considered instance. 

\begin{lemma}
	\label{lem:HQh}
Let $L_f>0$ and $L_A\ge0$.   For the instance of \eqref{eq:ECO_H} with data given in \eqref{eq:HQh}, it has a unique optimal solution $\vx^*$ given in \eqref{eq:sol-lam=0}, and there is a corresponding dual solution $\vy^*$ given by \begin{align}
	\label{eq:ECO_HQh_sol-lam=0-y}
	y_i^*=\begin{cases}
	-\frac{1}{2\sqrt{2}} & \text{ if }1\le i\le 2k
	\\
	0 & \text{ if }i\ge 2k+1.
	\end{cases}
	\end{align} In addition, the optimal objective value is
	\begin{align}
	\label{eq:ECO_HQh_fstar}
	f^* = -\left(\frac{L_f}{4} + \frac{L_A}{2\sqrt{2}}\right)k,
	\end{align}
%
	and the norm of the dual solution is 
	\begin{align}
	\label{eq:HQh_nrm-y-lam=0}\|\vy^*\| = \frac{\sqrt k}{2}.
	\end{align}
\end{lemma}
\begin{proof}
Similar to the proof of Lemma \ref{lem:ECO_Hh0_solutions}, we split $\vx$ into two parts as $\vx=(\vu^\top, \vv^\top)^\top$ with $\vu\in\RR^{2k}$ and $\vv\in\RR^{n-2k}$. Then from the block structure of $\vH$ and $\vA$ in \eqref{eq:HQh}, we obtain the following two optimization problems with respect to $\vu$ and $\vv$:
	\begin{align}\label{eq:ECO_H_1st-subprob}
	&\min_\vu \frac{1}{2}\vu^\top \vS \vu - \vs^\top\vu, \st \frac{L_A}{2}\vB\vu=\frac{L_A}{2}\vone_{2k},\\
	&\min_\vv \frac{L_f}{8}\|\vv\|^2, \st \frac{L_A}{2}\vG\vv=\vzero,\label{eq:ECO_H_1st-subprob-v}
	\end{align} 
	where 
	\begin{align*}
	\vS = \frac{L_f}{4} \vB^\top \vB\text{ and }\vs = \left(\frac{L_f}{4}+\frac{L_A}{4\sqrt{2}}\right) \ve_{2k,2k}. 
	\end{align*}

Since $L_f>0$, \eqref{eq:ECO_H_1st-subprob-v} clearly has a unique solution $\vv^*=\vzero$. If $L_A=0$, \eqref{eq:ECO_H_1st-subprob} is unconstrained, and $\vu^*$ is an optimal solution if and only if the optimality condition $\vS\vu^*=\vs$ holds. Note $L_f>0$ and $\vB$ is nonsingular. Then it is easy to verify that $\vu^*=(1,2,\ldots,2k)^\top$ is the unique point that satisfies the optimality condition and thus is an optimal solution to \eqref{eq:ECO_H_1st-subprob}. If $L_A>0$, then by \eqref{eq:ustar}, $\vu^*=(1,2,\ldots,2k)^\top$ is the unique feasible and thus optimal solution of \eqref{eq:ECO_H_1st-subprob}. Hence in both cases, we conclude $\vu^*$ is unique, and thus $\vx^*$ is unique and given in \eqref{eq:sol-lam=0}. Consequently,
	\begin{align*}
	f^* = \frac{1}{2}(\vu^*)^\top\vS\vu^* - \vs^\top \vu^* = \frac{L_f}{8}\|\vB\vu^*\|^2 - \vs^\top \vu^* = \frac{L_f}{8}\|\vone_{2k}\|^2 - \vs^\top \vu^* = -\frac{L_f}{4}k - \frac{L_A}{2\sqrt{2}}k.
	\end{align*}
	
To derive the corresponding dual variable,	 we split $\vy=(\vlam^\top,\vpi^\top)^\top$ 
	with $\vlam\in\RR^{2k}$ and $\vpi\in\RR^{m-2k}$. It follows from the KKT optimality conditions of \eqref{eq:ECO_H_1st-subprob} that
	\begin{equation}\label{eq:ECO_H_kkt-lam=0-reduced}
	\frac{L_A}{2}\vB^\top \vlam^* = \vS \vu^* - \vs, \quad \frac{L_A}{2}\vG^\top\vpi^* = \frac{L_f}{4}\vv^*=\vzero.
	\end{equation}
Obviously, $\vpi^*=\vzero$ is a solution of the above second equation. 
	In addition, from the definition of $\vB$ in \eqref{eq:B} and $\vu^*=(1,2,\ldots,2k)^\top$,  it is straightforward to verify
	\begin{align*}
	\vS\vu^* - \vs 
	= \frac{L_f}{4}\vB^\top\vB\vu^* - \vs
	= \frac{L_f}{4}\vB^\top\vone_{2k} - \vs
	= \frac{L_f}{4}\ve_{2k,2k} - \vs = -\frac{L_A}{4\sqrt{2}} \ve_{2k,2k}.
	\end{align*}
If $L_A=0$, then $\vlam^*=-\frac{1}{2\sqrt{2}}\vone_{2k}$ obviously satisfies the first equation in \eqref{eq:ECO_H_kkt-lam=0-reduced}. If $L_A>0$, we use \eqref{eq:invB}, the above result and \eqref{eq:ECO_H_kkt-lam=0-reduced} to have
	\begin{align*}
	\vlam^* = \frac{2}{L_A}\left(\vB^\top\right)^{-1}(-\frac{L_A}{4\sqrt{2}}\ve_{2k,2k}) = -\frac{1}{2\sqrt{2}}\vone_{2k},
	\end{align*}
Therefore, \eqref{eq:ECO_HQh_sol-lam=0-y} follows immediately, and it is straightforward to have $\|\vy^*\|=\frac{\sqrt k}{2}$.	
\end{proof}

\begin{remark}\label{rmk:HQh}
From the above proof, we see that if $L_A>0$, then the dual solution $\vy^*$ is also unique and given in \eqref{eq:ECO_HQh_sol-lam=0-y} since $\vG$ is full row rank.
\end{remark}

Using Lemma \ref{lem:HQh}, we have the following estimate.
\begin{lemma}
	\label{lem:ECO_HQh_KkBounds}
Let $L_f>0$ and $L_A\ge0$. Assume $L_f\ge L_A$. Then for the instance of \eqref{eq:ECO_H} with data given in \eqref{eq:HQh}, we have
	\begin{subequations}
		\begin{align}
		\label{eq:ECO_HQh_KkBounds_obj}
		\min_{\vx\in\cK_{k-1}} f(\vx) - f^*\ge & ~ \frac{3L_f\|\vx^*\|^2}{128(k+1)^2} + \frac{\sqrt{3}L_A\|\vx^*\|\cdot\|\vy^*\|}{8(k+1)}, 
		\\[0.1cm]
		\label{eq:ECO_HQh_KkBounds_feas}
		\min_{\vx\in\cK_{k-1}}\|\vA\vx - \vb\| \ge &~  \frac{\sqrt{3}L_A\|\vx^*\|}{4\sqrt{2}(k+1)},
		\end{align}
	\end{subequations}
	where $\vx^*$ is the unique primal solution given in \eqref{eq:sol-lam=0}, $\vy^*$ is a corresponding dual solution given in \eqref{eq:ECO_HQh_sol-lam=0-y}, 	and $\cK_{k-1}$ is defined in \eqref{eq:KJ}.
\end{lemma}

\begin{proof}
The result in \eqref{eq:ECO_HQh_KkBounds_feas} holds exactly the same as that in \eqref{eq:ECO_Hh0_KkBounds_feas}.

	To prove \eqref{eq:ECO_HQh_KkBounds_obj} we need to compute the minimal objective value of $f(\vx)$ over $\cK_{k-1}$. By \eqref{eq:krylov} we have $\cK_{k-1} = \Span\{\ve_{k+1,n},\ldots,\ve_{2k,n}\}$. Hence, for any $\vx\in\cK_{k-1}$, we can write it as $\vx=(\vzero_k^\top, \vz^\top, \vzero_{n-2k}^\top)^\top$ where $\vz\in \R^k$. Recalling \eqref{eq:HQh}, we have
	\begin{align*}
	\vh^\top\vx = & \left(\frac{L_f}{4} + \frac{L_A}{4\sqrt{2}}\right)x_{2k} = \left(\frac{L_f}{4} + \frac{L_A}{4\sqrt{2}}\right)z_k,
	\\
	\vx^\top\vH\vx  & = \frac{L_f}{4}\left\|\vB\begin{pmatrix}
	\vzero_k\\ \vz
	\end{pmatrix}\right\|^2 = \frac{L_f}{4}\|\bar\vB\vz\|^2 ,
	\end{align*}
	where
	\begin{align*}
	\bar\vB :=\left[\begin{array}{rrrrr}
	&  &  &-1 & \ 1\\
	& & \iddots & \iddots &\\
	& -1 & 1 &  & \\ 
	-1 & 1 &  & & \\
	1 &  &  & & 
	\end{array}\right]\in\RR^{k\times k}
	\end{align*}
	is a $k\times k$ submatrix of $\vB$. Therefore, 
	\begin{align}\label{eq:equiv-un-opt}
	\min_{\vx\in \cK_{k-1}}f(\vx) = \min_{\vz\in\R^k}\frac{L_f}{8}\|\bar\vB \vz\|^2 - \left(\frac{L_f}{4} + \frac{L_A}{4\sqrt{2}}\right)z_k.
	\end{align}
	Let $\vz^*$ be the optimal solution to the right hand side minimization problem in \eqref{eq:equiv-un-opt}. Then it must satisfy the optimality condition: 
	\begin{align*}
	\frac{L_f}{4}\bar\vB^2\vz^* =  \left(\frac{L_f}{4} + \frac{L_A}{4\sqrt{2}}\right)\ve_{k,k},
	\end{align*}
	which has the unique solution
	\begin{align*}
	\vz^* = \frac{4}{L_f}\left(\frac{L_f}{4} + \frac{L_A}{4\sqrt{2}}\right)(1,\ldots,k)^\top.
	\end{align*}
	Plugging $\vz=\vz^*$ into the right hand side of \eqref{eq:equiv-un-opt} yields 
	\begin{align}
	\label{eq:ECO_HQh_fmin_K}
	\min_{\vx\in \cK_{k-1}}f(\vx) 
	= -\frac{1}{8}\left(L_f + \sqrt{2}L_A + \frac{L_A^2}{2L_f}\right)k.
	\end{align}	
	From the above result and \eqref{eq:ECO_HQh_fstar}, we have
	\begin{align}\label{eq:bd-min-f-diff}
	\min_{\vx\in\cK_{k-1}}f(\vx) - f^* = \frac{1}{8}\left(L_f + \sqrt{2} L_A - \frac{L_A^2}{2L_f}\right)k \ge \frac{1}{8}\left(\frac{L_f}{2} + \sqrt{2} L_A\right)k.
	\end{align}
	where the inequality follows from $L_f\ge L_A$.
	Moreover, using \eqref{eq:ECO_Hh0_nrm-x-lam=0} and \eqref{eq:HQh_nrm-y-lam=0}, we have
	\begin{align*}
	\frac{L_f}{16}k + \frac{\sqrt{2}L_A}{8}k = & \frac{3L_f\|\vx^*\|^2}{16(2k+1)(4k+1)} + \frac{\sqrt{6}L_A\|\vx^*\|\cdot\|\vy^*\|}{4\sqrt{(2k+1)(4k+1)}}
	\\
	\ge & \frac{3L_f\|\vx^*\|^2}{128(k+1)^2} + \frac{\sqrt{3}L_A\|\vx^*\|\cdot\|\vy^*\|}{8(k+1)},
	\end{align*}
which together with \eqref{eq:bd-min-f-diff} gives \eqref{eq:ECO_HQh_KkBounds_obj} and completes the proof.
\end{proof}

Using Lemmas \ref{lem:ECO_HQh_xt-sp} through \ref{lem:ECO_HQh_KkBounds}, we are ready to establish the following lower complexity bound results.

\begin{theorem}[lower complexity bound II under linear span assumption]
	\label{thm:ECO_HQh_lb-complexity-lam=0}
Let $m\le n$ be positive integers, $L_f>0$, and $L_A\ge0$. Assume $L_f\ge L_A$. 	
For any positive integer $t< \frac{m}{2}$, there exists an instance of \eqref{eq:ECO} such that $\nabla f$ is $L_f$-Lipschitz continuous, $\|\vA\|=L_A$, and it has a primal-dual solution $(\vx^*,\vy^*)$. In addition, for any algorithm on solving \eqref{eq:ECO}, if it satisfies Assumption \ref{assum:linear_span}, then we have
	\begin{subequations}\label{eq:obj-feas-err-lam=0}
		\begin{align}
		f( \vx^{(t)} )-f( \vx ^*)\ge & ~ \frac{3L_f\|\vx^*\|^2}{128(t+1)^2} + \frac{\sqrt{3}L_A\|\vx^*\|\cdot\|\vy^*\|}{8(t+1)}, \label{eq:ECO_HQh_obj-err-lam=0} \\[0.1cm]
		\| \vA  \vx^{(t)} - \vb \| \ge &~  \frac{\sqrt{3}L_A\|\vx^*\|}{4\sqrt{2}(t+1)}.\label{eq:ECO_HQh_feas-err-lam=0}
		\end{align}
	\end{subequations}
\end{theorem}

\begin{proof} 
Set $k=t<\frac m 2$ and consider the instance \eqref{eq:ECO_H} with data given in \eqref{eq:HQh}. Clearly, this instance is in the form of \eqref{eq:ECO}, $\nabla f$ is $L_f$-Lipschitz continuous, and $\|\vA\|=L_A$. In addition, Lemma \ref{lem:HQh} indicates that it has a primal-dual solution $(\vx^*,\vy^*)$. By Lemma \ref{lem:ECO_HQh_xt-sp} and noting  $t=k$, we have $\vx^{(t)}\in \cK_{k-1}$. Consequently, 
	\begin{align*}
	f( \vx^{(t)} )-f( \vx ^*) \ge \min_{\vx\in\cK_{k-1}}f(\vx) - f^*,\text{ and }\| \vA  \vx^{(t)} - \vb \|\ge \min_{\vx\in\cK_{k-1}}\|\vA\vx - \vb\|.
	\end{align*}
	Now we conclude \eqref{eq:ECO_HQh_obj-err-lam=0}  and \eqref{eq:ECO_HQh_feas-err-lam=0} from Lemma \ref{lem:ECO_HQh_KkBounds}. 
	\end{proof}

\begin{remark}\label{rmk:ECO_HQh_lb-complexity-lam=0}
Let us compare the results in Theorems \ref{thm:ECO_Hh0_lb-complexity-lam=0} and \ref{thm:ECO_HQh_lb-complexity-lam=0}. First note that the former theorem requires $L_A>0$ while the latter one only needs $L_A\ge0$. As $L_A=0$, Theorem \ref{thm:ECO_HQh_lb-complexity-lam=0} recovers the lower complexity bound $O(\frac{L_f\|\vx^*\|^2}{(t+1)^2})$ for proximal gradient methods in \cite{nesterov2004introductory}. Second, the lower bounds for feasibility error are the same in the two theorems. For the objective error, both \eqref{eq:ECO_Hh0_obj-err-lam=0}
and \eqref{eq:ECO_HQh_obj-err-lam=0} have the term $\frac{L_A\|\vx^*\|\cdot\|\vy^*\|}{t+1}$. If this term dominates the other one in both inequalities, then the lower bounds in Theorems \ref{thm:ECO_Hh0_lb-complexity-lam=0} and \ref{thm:ECO_HQh_lb-complexity-lam=0} are in the same order, and their difference is that the former has an absolute value on the objective error while the latter one does not.
Thirdly, $O(\frac{L_f\|\vx^*\|^2}{(t+1)^2}+\frac{L_A\|\vx^*\|\cdot\|\vy^*\|}{t+1})$ has also appeared as an upper complexity bound for certain first-order methods; see \cite{ouyang2015accelerated, xu2017iter-complexity-ialm} and the inequality \eqref{eq:t2-ouyang} in section \ref{sec:tightness}. If $\|\vx^*\|$ and $\|\vy^*\|$ are regarded as constants, the term $\frac{L_f\|\vx^*\|^2}{(t+1)^2}$ can dominate $\frac{L_A\|\vx^*\|\cdot\|\vy^*\|}{t+1}$ for very big $t$ only when $L_f\gg L_A$. In fact, we observe that when $L_f\le L_A$, the result in \eqref{eq:ECO_HQh_obj-err-lam=0} reduces to $O(L_A/t)$. This observation is consistent with the result \eqref{eq:ECO_Hh0_obj-err-lam=0} in Theorem \ref{thm:ECO_Hh0_lb-complexity-lam=0}. 
Finally, it is interesting to note that in \eqref{eq:ECO_HQh_obj-err-lam=0}, $f( \vx^{(t)} )$ is always greater than $f^*$, and it is not necessarily the case in \eqref{eq:ECO_Hh0_obj-err-lam=0}, which uses the absolute value. \end{remark}

\subsection{A lower complexity bound for strongly convex case}
In this subsection, we develop a lower complexity bound for solving \eqref{eq:ECO} where $f$ is $\mu$-strongly convex, namely,
$$\langle \nabla f(\vx_1)-\nabla f(\vx_2), \vx_1-\vx_2\rangle \ge \mu\|\vx_1-\vx_2\|^2,\, \forall\, \vx_1,\vx_2\in\RR^n.$$ The measure we use is different from those in \eqref{eq:measures}. Instead of bounding the objective and feasibility error, we directly bound the distance of generated iterate to the unique optimal solution. Similar to the previous two subsections, the ``hard'' instance we design is also a quadratic program in the form of \eqref{eq:ECO_H}. The following theorem summarizes our result. 

\begin{theorem}[lower complexity bound for strongly convex case under linear span assumption]\label{thm:lb-complexity-lam=0-scvx}
	Let $m\le n$ be positive integers, $\mu>0$, and $L_A>0$. For any positive integer $t < \frac{m}{2}$, there exists an instance of \eqref{eq:ECO} such that $f$ is $\mu$-strongly convex, $\|\vA\|=L_A$, and it has a unique pair of primal-dual solution $( \vx ^*, \vy ^*)$. In addition, for any algorithm on solving \eqref{eq:ECO}, if it satisfies Assumption \ref{assum:linear_span}, then we have
	\begin{equation*}
	\|\vx^{(t)}-\vx^*\|^2 \ge \frac{5L_A^2 \|\vy^*\|^2}{256\mu^2 (t+1)^2}.
	\end{equation*}
\end{theorem}

\begin{proof}
	Set $k=t$ and consider an instance of \eqref{eq:ECO_H} with $\vH=\mu\vI$, $\vh=\vzero$, and $\vA$ and $\vb$ given in \eqref{eq:Ab}. Clearly $f$ is $\mu$-strongly convex, and $\|\vA\|=L_A$. It is easy to verify that Lemma \ref{lem:xt-sp} applies to this instance, and thus $\vx^{(t)}\in \cK_{t-1}$. Also, by writing the KKT condition, we can easily verify that the system has a unique primal-dual solution $( \vx ^*, \vy ^*)$ with $\vx^*$ given in \eqref{eq:sol-lam=0} and $\vy^*$ given by
	\begin{equation}\label{eq:sol-lam=0-y-scvx}
	y^*_i=\left\{\begin{array}{ll}\dfrac{\mu}{L_A}i(4k-i+1), & \text{ if } 1\le i \le 2k,\\[0.2cm]
	0, & \text{ if } i\ge 2k+1. 
	\end{array}\right.
	\end{equation}
	From the formula of $\cK_{i}$ in \eqref{eq:krylov}, it follows that for any $\vx\in\cK_{k-1}$,
	\begin{equation}\label{eq:bd-iter-diff}
	\|\vx-\vx^*\|^2\ge \sum_{i=1}^k i^2 \overset{\eqref{eq:sum-i2}}= \frac{k(k+1)(2k+1)}{6}.
	\end{equation}
	Moreover, by \eqref{eq:sum-i2} and also the formulas
	$$\sum_{i=1}^p i^3=\frac{p^2(p+1)^2}{4},\ \sum_{i=1}^p i^4=\frac{p(p+1)(2p+1)(3p^2+3p-1)}{30},$$
	we have from \eqref{eq:sol-lam=0-y-scvx} that
	\begin{align*}
	\|\vy^*\|^2= &~ \frac{\mu^2}{L_A^2}\sum_{i=1}^{2k} i^2(4k-i+1)^2\\
	=&~ \frac{\mu^2}{L_A^2} \left((4k+1)^2\sum_{i=1}^{2k} i^2-2(4k+1)\sum_{i=1}^{2k} i^3+\sum_{i=1}^{2k} i^4\right)\\
	=&~ \frac{2k(2k+1)(4k+1)\mu^2}{L_A^2}\left(\frac{(4k+1)^2}{6}-k(2k+1)+\frac{12k^2+6k-1}{30}\right)\\
	=&~ \frac{2k(2k+1)(4k+1)\mu^2}{15L_A^2}(16k^2+8k+2).
	\end{align*}
Since $t=k$ and $\vx^{(t)}\in \cK_{t-1}$, it is not difficult to verify the desired result from \eqref{eq:bd-iter-diff} and the above equation, and thus we complete the proof.
\end{proof}

\section{Lower complexity bounds of general first-order methods for affinely constrained problems}
\label{sec:ECO_general_case}

In this section, we drop the linear span assumption (see Assumption \ref{assum:linear_span}) and establish lower complexity bounds of general first-order methods described in \eqref{eq:I} on solving \eqref{eq:ECO}. The key idea is to utilize certain rotational invariance of quadratic functions and linear systems, a technique that was introduced in \cite{nemirovsky1991optimality,nemirovski1992information}. Specifically, we use the following proposition as a main tool and then derive the lower complexity bounds by the results obtained in the previous section. 

\begin{proposition}
	\label{pro:linear_span_ECO}
Let $m\le n$ and $k<\frac{m}{2}$ be positive integers, and let $L_f$ and $L_A$ be nonnegative numbers. Suppose that we have an instance of \eqref{eq:ECO_H}, called \emph{original instance}, where $\|\vH\|\le L_f$, and $\vA$ and $\vb$ are those given in \eqref{eq:Ab}. Moreover, assume that $\vH\in\SS_+^n$ and satisfies $\vH\cK_{2s-1}\subseteq\cK_{2s}$ for any $s\le \frac k 2$ and $\vh\in\cK_0$, where $\cK_{i}$ is defined in \eqref{eq:KJ}.
	Then we have the following results.
\begin{enumerate}
\item	For any first-order method $\cM$ that is described in \eqref{eq:I}, there exists  another problem instance, called \emph{rotated instance},
	\begin{align}
	\label{eq:ECO_rotate}
	\tilde f^*:=\min_{\vx\in \R^n}\big\{\tilde f(\vx), \st \tilde \vA \vx = \vb\big\},
	\end{align}
	such that $\nabla \tilde f$ is $L_f$-Lipschitz continuous, and $\|\tilde \vA\| = \|\vA\|$. More precisely, we have
	\begin{align}
		\label{eq:ECO_UV}
		\tilde f(\vx) = f(\vU\vx)\text{ and }\tilde\vA = \vV^\top \vA\vU,
	\end{align}
	where $\vU$ and $\vV$ are certain orthogonal matrices such that $\vU\vh = \vh$ and $\vV\vb = \vb$. 
\item	In addition, 
	$(\vx^*,\vy^*)$ is a primal-dual solution to the \emph{original instance} if and only if $(\hat\vx,\hat\vy):=(\vU^\top\vx^*,\vV^\top\vy^*)$ is a primal-dual solution to the \emph{rotated instance}. 	
\item Furthermore, when $\cM$ is applied to solve \eqref{eq:ECO_rotate}, for any $1\le t\le \frac{k-4}{2}$, we have:
	\begin{align}
	\label{eq:linear_span_lb-1}
	& \big|\tilde f(\xu^{(t)}) - \tilde f^*\big| \ge \min_{\vx\in\cK_{k-1}} \big|f(\vx) - f^*\big|, 
	\\
	& \tilde f(\xu^{(t)}) - \tilde f^* \ge \min_{\vx\in\cK_{k-1}} f(\vx) - f^*, \label{eq:linear_span_lb-2}
	\\
	& \|\tilde\vA\xu^{(t)} - \vb\|\ge \min_{\vx\in\cK_{k-1}}\|\vA\vx - \vb\|,
	\label{eq:linear_span_lb-3}
	\\
	& \|\xu^{(t)} - \hat\vx\|^2\ge \min_{\vx\in\cK_{k-1}}\|\vx - \vx^*\|^2,
	\label{eq:linear_span_lb-4}
	\end{align}
where $\xu^{(t)}$ is the computed approximate solution by $\cM$. 	
\end{enumerate}
\end{proposition}

The proof of Proposition \ref{pro:linear_span_ECO} is rather technical and deferred after we present the lower complexity bound results. Here we give a few remarks on this proposition.
First, in Proposition \ref{pro:linear_span_ECO} there are two problem instances, which have been distinguished  as \emph{original} and \emph{rotated} instances, respectively. Second, the results in \eqref{eq:linear_span_lb-1} through \eqref{eq:linear_span_lb-4} establish an important relation between the original and rotated instances. Specifically, by this relation, we are able to  study the best possible performance of general first-order methods through the linear subspace $\cK_{k-1}$. 

\subsection{Lower complexity bounds}
In this subsection, we apply Proposition \ref{pro:linear_span_ECO} together with Lemmas \ref{lem:ECO_Hh0_KkBounds} and  \ref{lem:ECO_HQh_KkBounds}, and Theorem \ref{thm:lb-complexity-lam=0-scvx} to establish the lower complexity bounds of general first-order methods on solving \eqref{eq:ECO}. The theorem below extends the results in Theorem \ref{thm:ECO_Hh0_lb-complexity-lam=0}, and Remarks \ref{rmk:ECO_Hh0_lb-complexity-lam=0} and \ref{rmk:ECO_HQh_lb-complexity-lam=0} also apply here. 

\begin{theorem}[lower complexity bound I of general first-order methods]
	\label{thm:ECO_Hh0_general}
	Let $8<m\le n$ be positive integers, and $L_f$ and $L_A$ be positive numbers. For any positive integer $t < \frac{m}{4}-2$ and any first-order method $\cM$ that is described in \eqref{eq:I}, there exists an instance of \eqref{eq:ECO_rotate} such that $\nabla \tilde f$ is $L_f$-Lipschitz continuous and $\|\tilde\vA\|=L_A$. In addition, the instance has a primal-dual solution $( \hat\vx, \hat\vy)$, and 
		\begin{align*}
		\big|\tilde f( \bar\vx^{(t)} )-\tilde f^*\big|\ge & ~ \frac{3L_f\|\hat\vx\|^2}{32(2t+5)} + \frac{\sqrt{6}}{32(2t+5)}L_A\|\hat\vx\|\cdot\|\hat\vy\|,  \\[0.1cm]
		\| \tilde\vA  \bar\vx^{(t)} - \vb \| \ge &~  \frac{\sqrt{3}L_A\|\hat\vx\|}{4\sqrt{2}(2t+5)},
				\end{align*}
	where $ \bar\vx^{(t)} $ is the approximate solution output by $\cM$.		
\end{theorem}

\begin{proof}
	Set $k=2t+4 < \frac{m}{2}$ in the definition of $\vLam$ and $\vc$ given in \eqref{eq:Lamc}. Consider the problem instance \eqref{eq:ECO_Hh0} and let it be the \emph{original} instance. It is easy to check that the data satisfy the conditions required in Proposition \ref{pro:linear_span_ECO}. Hence, by items 1 and 3 of Proposition \ref{pro:linear_span_ECO}, there exists a \emph{rotated} instance \eqref{eq:ECO_rotate} such that \eqref{eq:linear_span_lb-1} and  \eqref{eq:linear_span_lb-3} hold. Applying Lemma \ref{lem:ECO_Hh0_KkBounds} together with these two inequalities, we have
	\begin{subequations}\label{eq:ECO_obj-feas-err-lam=0-sp-tmp}
	\begin{align}
		\big|\tilde f(\xu^{(t)}) - \tilde f^*\big| \ge & ~\min_{\vx\in\cK_{k-1}}\big|f(\vx) - f^*\big|\ge  \frac{3L_f\|\vx^*\|^2}{32(k+1)} + \frac{\sqrt{6}}{32(k+1)}L_A\|\vx^*\|\cdot\|\vy^*\|, 
		\\[0.1cm]
		\|\tilde\vA\xu^{(t)} - \vb\|\ge & ~\min_{\vx\in\cK_{k-1}}\|\vA\vx - \vb\| \ge  \frac{\sqrt{3}L_A\|\vx^*\|}{4\sqrt{2}(k+1)},
		\end{align}
	\end{subequations}
	where $(\vx^*,\vy^*)$ is the unique pair of primal-dual solution to the original instance \eqref{eq:ECO_Hh0}. By item 2 of Proposition \ref{pro:linear_span_ECO}, the rotated instance \eqref{eq:ECO_rotate} also has a unique primal-dual solution $(\hat\vx,\hat\vy)$ given by $\hat\vx = \vU^\top\vx^*$ and $\hat\vy = \vV^\top\vy^*$. Since $\vU$ and $\vV$ are orthogonal, it holds that $\|\vx^*\| = \|\hat\vx\|$ and $\|\vy^*\| = \|\hat\vy\|$. Therefore, noting that $k=2t+4$, we obtain the desired results from the two inequalities in \eqref{eq:ECO_obj-feas-err-lam=0-sp-tmp} and complete the proof.
\end{proof}

In Theorem \ref{thm:ECO_Hh0_general}, $L_A>0$ is required. The following theorem allows $L_A=0$ and extends the results in Theorem \ref{thm:ECO_HQh_lb-complexity-lam=0} to general first-order methods. 
\begin{theorem}[lower complexity bound II of general first-order methods]
	\label{thm:ECO_HQh_general}
	Let $8<m\le n$ be positive integers, $L_f>0$, and $L_A\ge0$. Assume $L_f\ge L_A$. For any positive integer $t < \frac{m}{4}-2$ and any first-order method $\cM$ that is described in \eqref{eq:I}, there exists an instance of \eqref{eq:ECO_rotate} such that $\nabla \tilde f$ is $L_f$-Lipschitz continuous and $\|\tilde\vA\|=L_A$. In addition, the instance has a primal-dual solution $(\hat{\vx},\hat{\vy})$, and
	\begin{subequations}
		\label{eq:obj-feas-err-lam=0-gen}
		\begin{align}
		\tilde f( \xu^{(t)} )-\tilde f^*\ge & ~ \frac{3L_f\|\hat\vx\|^2}{128(2t+5)^2} + \frac{\sqrt{3}L_A\|\hat\vx\|\cdot\|\hat\vy\|}{8(2t+5)}, 
		\label{eq:ECO_HQh_obj-err-lam=0-gen} 
		\\[0.1cm]
		\| \tilde \vA  \xu^{(t)} - \vb \| \ge &~  \frac{\sqrt{3}L_A\|\hat\vx\|}{4\sqrt{2}(2t+5)},
		\label{eq:ECO_HQh_feas-err-lam=0-gen}
		\end{align}
	\end{subequations}
	where $ \bar\vx^{(t)} $ is the approximate solution output by $\cM$. 
\end{theorem}

For strongly convex case, we below generalize Theorem \ref{thm:lb-complexity-lam=0-scvx} to any first-order method given in \eqref{eq:I}.

\begin{theorem}[lower complexity bound of general first-order methods for strongly convex case]
	\label{thm:ECO_scvx_general}
	Let $8<m\le n$ be positive integers, and $\mu$ and $L_A$ be positive numbers. For any positive integer $t < \frac{m}{4}-2$ and any first-order method $\cM$ that is described in \eqref{eq:I}, there exists an instance of \eqref{eq:ECO_rotate} such that $\tilde f$ is $\mu$-strongly convex, and $\|\tilde\vA\|=L_A$. In addition, the instance has a unique primal-dual solution $( \hat\vx, \hat\vy)$, and
	\begin{equation}\label{eq:lbd-complexity-scvx-sp}
	\|\xu^{(t)}-\hat\vx\|^2 \ge \frac{5L_A^2 \|\vy^*\|^2}{256\mu^2 (2t+5)^2},
	\end{equation}
	where $ \bar\vx^{(t)} $ is the approximate solution of output by $\cM$.
\end{theorem}

The proofs of Theorems \ref{thm:ECO_HQh_general} and \ref{thm:ECO_scvx_general} are similar to that of Theorem \ref{thm:ECO_Hh0_general}. To show Theorem \ref{thm:ECO_HQh_general}, one can apply \eqref{eq:linear_span_lb-2} and \eqref{eq:linear_span_lb-3} together with Lemma \ref{lem:ECO_HQh_KkBounds}, and for Theorem \ref{thm:ECO_scvx_general}, one can use \eqref{eq:linear_span_lb-4} together with Theorem \ref{thm:lb-complexity-lam=0-scvx}. We omit the details.

\subsection{Proof of Proposition \ref{pro:linear_span_ECO}}

This subsection is dedicated to the technical details on the proof of Proposition \ref{pro:linear_span_ECO}. In section \ref{sec:lb-spp}, a similar proposition (i.e., Proposition \ref{pro:linear_span_NCO}) will be shown for the SPPs. Since an affinely constrained problem can be equivalently formulated as one SPP, we conduct the analysis directly on instances of SPP. For ease of notation, we define a specific class of SPPs as follows.  
\begin{definition}\label{def:ptheta-HA}
Given $\vH\in\SS_+^{n}$, $\vA\in\RR^{m\times n}$, and $\vtheta=(\vh, \vb, X, Y, 0)$, $P(\vtheta; \vH,\vA)$ denotes as one instance of \eqref{eq:SPP} with $f(\vx)=\frac{1}{2}\vx^\top \vH\vx - \vh^\top\vx$ and $g=0$. 
\end{definition}
If $X=\RR^n$, and $Y=\RR^m$, then $P(\vtheta;\vH,\vA)$ is an instance of \eqref{eq:ECO_H}. On an instance $P(\vtheta;\vH,\vA)$, the first-order method $\cM$ described in \eqref{eq:I} can be written as 
\begin{align}
\label{eq:I_QP}
\left(\vx^{(t+1)}, \vy^{(t+1)}, \xu^{(t+1)}, \yu^{(t+1)}\right) = \cI_t\left(\vtheta;\vH\vx^{(0)},\vA\vx^{(0)},\vA^\top \vy^{(0)},\ldots,\vH\vx^{(t)},\vA\vx^{(t)},\vA^\top \vy^{(t)}\right),\,\forall\, t\ge 0.
\end{align}

We start our proof with several technical lemmas. 
The following lemma is an elementary result of linear subspaces and will be used several times in our analysis.
\begin{lemma}
	\label{lem:ortho_map}
	Let $\cX\subsetneq \bar \cX\subseteq\RR^p$ be two linear subspaces. Then for any $\bar\vx\in \RR^p$, there exists an orthogonal matrix $\vV\in\RR^{p\times p}$ such that
	\begin{align}
	\label{eq:lemU}
	\vV\vx = \vx,\ \forall \vx\in \cX,\text{ and }\vV\bar\vx\in\bar \cX.
	\end{align}
	\end{lemma}
	
	\begin{proof}
		If $\bar\vx\in \cX$, then we can simply choose $\vV=\vI$. Otherwise, we decompose $\bar\vx = \vy+\vz$, where $\vzero\neq\vy\in \cX^\perp$ and $\vz\in \cX$. Let $s=\dim(\cX)$ and $t=\dim(\bar\cX)>s$. Assume $\vu_1,\ldots,\vu_s$ to be an orthonormal basis of $\cX$. We extend it to $\vu_1,\ldots,\vu_t$, an orthonormal basis of $\bar \cX$. The desired result in \eqref{eq:lemU} is then obtained by choosing $\vV$ as an orthogonal matrix such that $\vV\vu_i=\vu_i,\,\forall\,i=1,\ldots,s$, and $\vV\vy = \|\vy\|\vu_{s+1}$.  
	\end{proof}

By Lemma \ref{lem:ortho_map}, we show the results below.

\begin{lemma}\label{lem:PhiPsi-cond}
Given $m\le n$ and $k < \frac{m}{2}$, let $\vLam$ be the matrix in \eqref{eq:Lamc}. Let $s\le \frac k 2$ be a positive integer, $\vH\in\SS_+^n$, and $\vU, \vPhi\in\RR^{n\times n}$ and $\vV,\vPsi\in\RR^{m\times m}$ be orthogonal matrices. If $\vH\cK_{2s-1}\subseteq \cK_{2s}$, and
\begin{equation}\label{eq:PhiPsi-cond}
\vPhi \vx = \vx, \forall\, \vx\in \vU^\top \cK_{2s},~ \mathrm{and}~\vPsi \vy = \vy,\, \forall\, \vy\in \vV^\top \cJ_{2s},
\end{equation}
then for any $ \vx\in \vU^\top \cK_{2s-1}$ and any $\vy\in \vV^\top \cJ_{2s-1}$, it holds:
\begin{equation*}
		 \tilde{\vU}^\top \vH\tilde{\vU} \vx = \vU^\top \vH\vU\vx,\ \tilde{\vV}^\top \vLam\tilde{\vU} \vx = \vV^\top \vLam\vU \vx, \text{ and }\tilde{\vU}^\top \vLam^\top \tilde{\vV} \vy = \vU^\top \vLam^\top \vV \vy.
		\end{equation*}
where $\tilde{\vU}=\vU\vPhi$ and $\tilde{\vV}=\vV\vPsi$.		
\end{lemma}

\begin{proof}
Let $\vx\in \vU^\top \cK_{2s-1}$ and $\vy\in \vV^\top \cJ_{2s-1}$. Since $\vU$ and $\vV$ are orthogonal, it holds that $\vU \vx \in \cK_{2s-1}$ and $\vV\vy\in \cJ_{2s-1}$. Hence, from the assumption on $\vH$, the properties of $\cJ_i$ and $\cK_i$ in \eqref{eq:AonKJ} and \eqref{eq:KJexpan}, and noting $2s-1\le k-1$, we have
		\begin{align*}
		\vH\vU\vx \in\cK_{2s},\ \vLam\vU \vx\in \cJ_{2s},\text{ and }\vLam^\top\vV\vy\in \cK_{2s-1} \subsetneq \cK_{2s},
		\end{align*}
		which implies 
		$$\vU^\top\vH\vU\vx \in\vU^\top\cK_{2s},\ \vV^\top \vLam\vU \vx\in \vV^\top \cJ_{2s},\text{ and }\vU^\top \vLam^\top\vV \vy\in \vU^\top \cK_{2s}.$$ From \eqref{eq:PhiPsi-cond}, we obtain 
$$\vPhi\vU^\top\vH\vU\vx  = \vU^\top\vH\vU\vx,\ \vPsi  \vV^\top \vLam\vU\vx = \vV^\top \vLam\vU\vx, \text{ and }\vPhi  \vU^\top \vLam^\top \vV \vy = \vU^\top \vLam^\top \vV \vy.$$
Because $\vPhi$ and $\vPsi$ are orthogonal matrix, the above equations indicate	that
		\begin{align}
		\label{eq:tmp2}
		\vPhi^\top\vU^\top\vH\vU\vx  = \vU^\top\vH\vU\vx,\ \vPsi^\top  \vV^\top \vLam\vU\vx = \vV^\top \vLam\vU\vx, \text{ and }\vPhi^\top  \vU^\top \vLam^\top \vV \vy = \vU^\top \vLam^\top \vV \vy.
		\end{align}
		
		Moreover, since $\vx\in \vU^\top \cK_{2s-1}$ and $\vy\in \vV^\top \cJ_{2s-1}$, it follows from \eqref{eq:KJexpan} that $\vx\in \vU^\top \cK_{2s}$ and $\vy\in\vV^\top \cJ_{2s}$, and thus using \eqref{eq:PhiPsi-cond} again and also the definition of $\tilde\vU$ and $\tilde\vV$, we have 
		\begin{align}
		\label{eq:tmp3}
		\tilde\vU\vx = \vU\vPhi \vx = \vU\vx, \text{ and }\tilde\vV\vy = \vV\vPsi \vy = \vV\vy.
		\end{align}
		Therefore, we conclude that for any $ \vx\in \vU^\top \cK_{2s-1}$ and $\vy\in \vV^\top \cJ_{2s-1}$,
		\begin{align*}
		& \tilde\vU^\top \vH\tilde\vU \vx \overset{\eqref{eq:tmp3}}= \vPhi^\top\vU^\top \vH\vU\vx \overset{\eqref{eq:tmp2}} = \vU^\top \vH\vU\vx,
		\\
		& \tilde\vV^\top \vLam\tilde\vU \vx \overset{\eqref{eq:tmp3}}= \vPsi^\top \vV^\top \vLam\vU \vx \overset{\eqref{eq:tmp2}}= \vV^\top \vLam\vU \vx,
		\\
		& \tilde\vU^\top \vLam^\top \tilde\vV \vy = \vPhi^\top \vU^\top \vLam^\top \vV \vy \overset{\eqref{eq:tmp2}}= \vU^\top \vLam^\top \vV \vy.
		\end{align*}
Hence, we complete the proof.
\end{proof}

\begin{lemma}\label{lem:krylov}
Given $m\le n$ and $k < \frac{m}{2}$, let $\vLam$ and $\vc$ be the matrix and vector in \eqref{eq:Lamc}, and let $\vh\in\cK_0$.
Suppose that $\vA$ and $\vb$ are respectively a multiple of $\vLam$ and $\vc$ and $\vH\in\SS_+^n$ satisfying $\vH\cK_{2s-1}\subseteq\cK_{2s}$ for all $s\le \frac k 2$. Assume $X$ and $Y$ are rotational invariant convex sets. 
Then for any first-order method $\cM$ described in \eqref{eq:I}, there exist orthogonal matrices $\vU\in\RR^{n\times n}$ and $\vV\in\RR^{m\times m}$ and a problem instance $P(\vtheta;\vU^\top \vH\vU,\vV^\top \vA\vU)$ with $\vtheta=(\vh,\vb,X,Y,0)$ such that $\vU\vh = \vh$, $\vV\vc=\vc$, and in addition, for any $0\le t \le \frac{k-4}{2}$, when $\cM$ is applied to solve the instance, the iterates $\{(\vx^{(i)},\vy^{(i)})\}_{i=0}^t$ satisfy
	\begin{align*}
	\vx^{(i)}\in \vU^\top \cK_{2t+1},\ \vy^{(i)}\in \vV^\top \cJ_{2t+1},\ \forall\, i=0,\ldots,t,
	\end{align*} 
	where $\cK_{2t+1}$ and $\cJ_{2t+1}$ are the Krylov spaces defined in \eqref{eq:KJ}.
\end{lemma}

\begin{proof} 
Note $\cK_0\subsetneq\cK_1$ and $\cJ_0\subsetneq\cJ_1$ from Lemma \ref{lem:krylov-JK}. Hence, by Lemma \ref{lem:ortho_map} there exist orthogonal matrices $\vU_0$ and $\vV_0$ such that 
\begin{align*}
	\vU_0\vx = \vx, \forall \vx\in\cK_0,\text{ and }\vU_0\vx^{(0)} \in \cK_1
	\\
	\vV_0\vy = \vy, \forall \vy\in\cJ_0,\text{ and }\vV_0\vy^{(0)} \in \cJ_1.
\end{align*} 
Therefore, from the condition $\vh\in\cK_0$ and $\vc\in\cJ_0$ by Lemma \ref{lem:krylov-JK}, we have $\vU_0\vh = \vh$ and $\vV_0\vc = \vc$. Consequently, the results in the lemma hold for $t=0$.
Below we prove the results for any $t < \frac{k-4}{2}$ by induction. 

Assume that for some $1\le s< \frac{k-4}{2}$, the results hold for $t=s-1$, namely, there exist orthogonal matrices $\vU_{s-1}\in\RR^{n\times n}$ and $\vV_{s-1}\in\RR^{m\times m}$ such that $\vU_{s-1}\vh = \vh$, $\vV_{s-1}\vc=\vc$, and when $\cM$ is applied to the instance $P(\vtheta;\vU_{s-1}^\top \vH\vU_{s-1},\vV_{s-1}^\top \vA\vU_{s-1})$, the iterates $\{(\vx^{(i)},\vy^{(i)})\}_{i=0}^{s-1}$ satisfy
		\begin{align}
		\label{eq:hypothesis}
		\vx^{(i)}\in \vU_{s-1}^\top \cK_{2s-1},\text{ and }\vy^{(i)}\in \vV_{s-1}^\top \cJ_{2s-1},\ \forall i=0,\ldots,s-1.
		\end{align}
		Suppose the next iterate obtained from $\cM$ is $(\vx^{(s)},\vy^{(s)})\in X\times Y$. 
		Since $s< \frac{k-4}{2}$, it holds that $2s < k$, and from \eqref{eq:KJexpan} we have $\vU_{s-1}^\top \cK_{2s-1}\subsetneq \vU_{s-1}^\top \cK_{2s}\subsetneq \vU_{s-1}^\top \cK_{2s+1}$ and $\vV_{s-1}^\top \cJ_{2s-1}\subsetneq \vV_{s-1}^\top \cJ_{2s}\subsetneq \vV_{s-1}^\top \cJ_{2s+1}$.
		By Lemma \ref{lem:ortho_map}, there exist orthogonal matrices $\vPhi\in\RR^{n\times n}$ and $\vPsi\in\RR^{m\times m}$ such that 
		\begin{align}
		\label{eq:PhiPsi}
		\begin{aligned}
		&\vPhi \vx = \vx, \ \forall \vx\in \vU_{s-1}^\top \cK_{2s},\text{ and } \vPhi \vx^{(s)} \in \vU_{s-1}^\top \cK_{2s+1}, \\
		&\vPsi \vy = \vy,\ \forall \vy\in \vV_{s-1}^\top \cJ_{2s},\text{ and } \vPsi \vy^{(s)} \in \vV_{s-1}^\top \cJ_{2s+1}.
		\end{aligned}
		\end{align}
		Since $\vc\in\cJ_{2s}$ and $\vV_{s-1}\vc=\vc$, we have $\vc\in \vV_{s-1}^\top \cJ_{2s}$, and thus it follows from \eqref{eq:PhiPsi} that $\vPsi \vc=\vc$. 
		Let $\vU_s = \vU_{s-1}\vPhi$ and $\vV_s = \vV_{s-1}\vPsi$. Clearly, both $\vU_s$ and $\vV_s$ are orthogonal matrices, and because $\vV_{s-1}\vc=\vc$ and $\vPsi \vc=\vc$, we have $\vV_s\vc=\vc$. By a similar argument we also have $\vU_s\vh = \vh$. In addition, from \eqref{eq:PhiPsi}, Lemma \ref{lem:PhiPsi-cond}, and the assumptions on $\vH$ and $\vA$, it follows that for any $\vx\in \vU_{s-1}^\top \cK_{2s-1}$ and $\vy\in \vV_{s-1}^\top \cJ_{2s-1}$,
		\begin{align}
		\label{eq:QAinfo_same}
		\begin{aligned}
		& \vU_s^\top \vH\vU_s \vx = \vU_{s-1}^\top \vH\vU_{s-1}\vx,\ \vV_s^\top \vA\vU_s \vx = \vV_{s-1}^\top \vA\vU_{s-1} \vx, \text{ and }\vU_s^\top \vA^\top \vV_s \vy = \vU_{s-1}^\top \vA^\top \vV_{s-1} \vy.
		\end{aligned}
		\end{align}

Therefore, from the induction hypothesis \eqref{eq:hypothesis} and the equations in \eqref{eq:QAinfo_same}, we conclude that the first $s+1$ iterates obtained from $\cM$ applied to $P(\vtheta;\vU_s^\top \vH\vU_s,\vV_s^\top \vA\vU_s)$ are exactly the same as the first $s+1$ iterates obtained from $\cM$ applied to $P(\vtheta;\vU_{s-1}^\top \vH\vU_{s-1}\vV_{s-1}^\top \vA\vU_{s-1})$, because exactly the same information is used to generate those iterates (cf. \eqref{eq:I_QP}). Consequently, when $\cM$ is applied to $P(\vtheta;\vU_s^\top \vH\vU_s,\vV_s^\top \vA\vU_s)$, the first $s+1$ iterates are $(\vx^{(i)},\vy^{(i)}), i=0,1,\ldots,s$. Hence, from \eqref{eq:KJexpan}, \eqref{eq:hypothesis} and \eqref{eq:PhiPsi}, and also the facts $\vU_s = \vU_{s-1}\vPhi$ and $\vV_s = \vV_{s-1}\vPsi$, we have 
$$\vx^{(i)}\in \vU_s^\top \cK_{2s+1},\ \vy^{(i)}\in \vV_s^\top \cJ_{2s+1},\ \forall i=0,\ldots,s.$$
This finishes the induction and completes the proof.
\end{proof}

From Lemma \ref{lem:krylov}, we perform another rotation to the instance and can further show that the generated approximate solution falls into a rotated Krylov subspace.
\begin{proposition}
	\label{pro:krylov}
Given $m\le n$ and $k < \frac{m}{2}$, let $\vLam$ and $\vc$ be the matrix and vector in \eqref{eq:Lamc}, and let $\vh\in\cK_0$. 
Suppose that $\vA$ and $\vb$ are respectively a multiple of $\vLam$ and $\vc$ and $\vH\in\SS_+^n$ satisfying $\vH\cK_{2s-1}\subseteq\cK_{2s}$ for all $s\le \frac k 2$. Assume $X$ and $Y$ are rotational invariant convex sets. Then for any first-order method $\cM$ described in \eqref{eq:I}, there exist orthogonal matrices $\vU\in\RR^{n\times n}$ and $\vV\in\RR^{m\times m}$ and a problem instance $P(\vtheta;\vU^\top \vH\vU,\vV^\top \vA\vU)$ with $\vtheta=(\vh,\vb,X,Y,0)$ such that $\vU\vh = \vh$, $\vV\vc=\vc$, and in addition, for any $0\le t \le \frac{k-4}{2}$, when $\cM$ is applied to this instance, the iterates $\{(\vx^{(i)},\vy^{(i)})\}_{i=0}^t$ satisfy
$$\vx^{(i)}\in \vU^\top \cK_{2t+3},\ \vy^{(i)}\in \vV^\top \cJ_{2t+3},\ \forall\, i=0,\ldots,t,$$
	and the output $\xu^{(t)}\in \vU^\top \cK_{2t+3}$.
\end{proposition}
	
	\begin{proof}	
	From Lemma \ref{lem:krylov}, we see that there are orthogonal matrices $\vU_t\in\RR^{n\times n}$ and $\vV_t\in\RR^{m\times m}$	and a problem instance $P(\vtheta;\vU_t^\top \vH\vU_t,\vV_t^\top \vA\vU_t)$, such that $\vU_t\vh = \vh$,  $\vV_t\vc=\vc$, and when $\cM$ is applied to this instance, the iterates $\{(\vx^{(i)},\vy^{(i)})\}_{i=0}^t$ satisfy
	\begin{align}
	\label{eq:krylov_xy-2}
	\vx^{(i)}\in \vU_t^\top \cK_{2t+1},\ \vy^{(i)}\in \vV_t^\top \cJ_{2t+1},\ \forall\, i=0,\ldots,t.
	\end{align}
				
		Since $t\le \frac{k-4}{2}$, it holds that $2t+3\le k-1$, and from \eqref{eq:KJexpan} we have $\cK_{2t+1}\subsetneq \cK_{2t+2}\subsetneq \cK_{2t+3}$ and $\cJ_{2t+1}\subsetneq \cJ_{2t+2}\subsetneq \cJ_{2t+3}$. Hence, by Lemma \ref{lem:ortho_map}, there is an orthogonal matrices $\vPhi$ and $\vPsi$ such that
		\begin{align}
		\label{eq:tPhitPsi}
		\begin{aligned}
		&\vPhi \vx = \vx,\, \forall \vx\in \vU_{t}^\top \cK_{2t+2}, \quad \vPhi \xu^{(t)} \in \vU_{t}^\top \cK_{2t+3},\\
		&\vPsi \vy = \vy, \,\forall \vy\in \vV_{t}^\top \cJ_{2t+2}, \quad \vPsi \yu^{(t)} \in \vV_{t}^\top \cJ_{2t+3}.
		\end{aligned}
		\end{align}
		Let $\vU = \vU_t\vPhi$ and $\vV=\vV_t\vPsi$. By the same arguments as those in the proof of Lemma \ref{lem:krylov}, we have that $\vU\vh = \vh$, $\vV\vc = \vc$, and when $\cM$ is applied to $P(\vtheta; \vU^\top \vH\vU,\vV^\top \vA\vU)$, the first $t+1$ iterates are exactly $(\vx^{(i)},\vy^{(i)}),\, i=0,1,\ldots,t$. The desired results now follow from \eqref{eq:krylov_xy-2} and \eqref{eq:tPhitPsi}. 
	\end{proof}

Using Proposition \ref{pro:krylov}, we are now ready to prove Proposition \ref{pro:linear_span_ECO}.

\vgap

\begin{proof}[ of Proposition \ref{pro:linear_span_ECO}]
First note that the original instance in Proposition \ref{pro:linear_span_ECO} is $P(\vtheta; \vH, \vA)$ with $\vtheta=(\vh,\vb,\RR^n,\RR^m,0)$. Second, the data $\vH,\vA,\vb$ and $\vh$ satisfy the conditions in Proposition \ref{pro:krylov}. Hence, we apply Proposition \ref{pro:krylov} to
obtain an instance $P(\vtheta; \vU^\top\vH\vU, \vV^\top\vA\vU)$ such that $\vU\vh = \vh$ and $\vV\vb = \vb$, where $\vU$ and $\vV$ are orthogonal matrices, and we have used the fact that $\vb$ is a multiple of $\vc$. In addition, note that the  instance $P(\vtheta; \vU^\top\vH\vU, \vV^\top\vA\vU)$ is 
	\begin{align}\label{eq:rate-instance-H}
		\min_{\vx\in\R^n}\frac{1}{2}\vx^\top \vU^\top\vH\vU\vx - \vh^\top \vx\ \st \vV^\top\vA\vU\vx = \vb,
	\end{align}
	which is exactly the rotated instance \eqref{eq:ECO_rotate} by the definition \eqref{eq:ECO_UV} and the relations $\vU\vh = \vh$ and $\vV\vb = \vb$. By the KKT conditions of the original instance \eqref{eq:ECO_H} and the rotated instance \eqref{eq:ECO_rotate} or \eqref{eq:rate-instance-H}, it is easy to show that the pair $(\vx^*,\vy^*)$ is a primal-dual solution to the original instance if and only if $(\hat\vx,\hat\vy)=(\vU^\top\vx^*,\vV^\top\vy^*)$ is a primal-dual solution to the rotated instance, and that $\tilde f^* = f^*$.
	
	It remains to prove the inequalities from \eqref{eq:linear_span_lb-1} through \eqref{eq:linear_span_lb-4}. By Proposition \ref{pro:krylov}, when $\cM$ is applied to the rotated instance, the approximate solution $\xu^{(t)}\in \vU^\top \cK_{2t+3}$, which indicates $\vU\xu^{(t)}\in\cK_{2t+3}$ by the orthogonality of $\vU$. Since $t\le \frac{k-4}{2}$, we have $2t+3\le k-1$, and thus from \eqref{eq:KJexpan}, it follows that $\vU\xu^{(t)}\in \cK_{2t+3}\subseteq \cK_{k-1}$. Therefore, from \eqref{eq:ECO_UV} and the relations $\vU\vh = \vh$ and $\vV\vb = \vb$ we have
	\begin{align*}
	\begin{aligned}
	& \big|\tilde f(\xu^{(t)}) - \tilde f^*\big| = \big|f(\vU\xu^{(t)}) - f^*\big| \ge \min_{\vx\in\cK_{k-1}} \big|f(\vx) - f^*\big|, 
	\\
	& \tilde f(\xu^{(t)}) - \tilde f^* = f(\vU\xu^{(t)}) - f^*  \ge \min_{\vx\in\cK_{k-1}} f(\vx) - f^*,
	\\
	& \|\tilde\vA\xu^{(t)} - \vb\| = \|\vA(\vU\xu^{(t)}) - \vb\| \ge \min_{\vx\in\cK_{k-1}}\|\vA\vx - \vb\|,
	\\
	& \|\xu^{(t)} - \hat\vx\|^2 = \|\vU\xu^{(t)} - \vx^*\|^2 \ge \min_{\vx\in\cK_{k-1}}\|\vx - \vx^*\|^2,
	\end{aligned}
	\end{align*}
and we complete the proof.	
\end{proof}

\section{Lower complexity bounds on bilinear saddle-point problems}\label{sec:lb-spp}
In this section, we derive lower complexity bounds of first-order methods on solving the bilinear saddle-point problem \eqref{eq:saddle-prob} through considering its associated primal problem \eqref{eq:SPP}. As we mentioned in the beginning, the affinely constrained problem \eqref{eq:ECO} is a special case of \eqref{eq:SPP} if $Y=\RR^m$ and $g=0$. Hence, the results obtained in previous sections also apply to \eqref{eq:SPP}, namely, our designed instances of \eqref{eq:ECO} are also ``hard'' instances of \eqref{eq:SPP}. However, if we require $Y$ to be a compact set, those instances will not satisfy. On solving \eqref{eq:saddle-prob} with both $X$ and $Y$ being compact, \cite{nesterov2005smooth} gives a first-order method that can be described as \eqref{eq:I},  and it proves 
\begin{equation}\label{eq:rate-nesterov-spp}
0\le \phi(\bar{\vx}^{(t)})-\psi(\bar{\vy}^{(t)}) \le \frac{4L_f D_X^2}{(t+1)^2}+\frac{4\|\vA\|_{1,2}D_XD_Y}{t+1},
\end{equation}
where $D_X$ and $D_Y$ are the diameters of $X$ and $Y$ respectively, and
$$\|\vA\|_{1,2}:=\max_{\|\vx\|_1=1,\|\vy\|=1} \langle \vA\vx,\vy\rangle.$$
It is an open question if the convergence rate in \eqref{eq:rate-nesterov-spp} can still be improved. We give instances below to show a lower complexity bound, which is in the same form as that in \eqref{eq:rate-nesterov-spp} and differs only at the constants, and thus the result in \cite{nesterov2005smooth} is optimal. The ingredients in the designed ``hard'' SPP instances are the same as those used in section \ref{sec:linear_span_case}.

\subsection{A special class of SPP instances and a key proposition}
Similar to the previous two sections that focus on quadratic programs, we consider the SPP \eqref{eq:saddle-prob} with $f$ being a convex quadratic function and $g=0$, and we derive the lower complexity bounds through the primal problem \eqref{eq:SPP}. More precisely, let $m\le n$ and $k<\frac m 2$ be positive integers. We build instances of \eqref{eq:SPP} that are in the form of
\begin{align}
\label{eq:NCOspecial}
\phi^*:=\min_{\vx\in X}\Set{\phi(\vx):=f(\vx) + \max_{\vy\in Y}\langle \vb - \vA\vx, \vy\rangle},
\end{align}
where $X$ and $Y$ are compact Euclidean balls, $\vA$ and $\vb$ are those in \eqref{eq:Ab} with a certain $L_A\ge 0$, and $f(\vx)=\frac{1}{2}\vx^\top \vH\vx - \vh^\top\vx$ with $\vH\in\SS_+^n$. 

On affinely constrained problems, we first derive lower complexity bounds of first-order methods under linear span assumption in section \ref{sec:linear_span_case} for ease of the readers' understanding. The more technical derivation for general first-order methods is then given in section \ref{sec:ECO_general_case}. With those preparations, we now develop the lower complexity bounds directly on the general first-order methods described in \eqref{eq:I}. The following proposition is an extension of Proposition \ref{pro:linear_span_ECO} and serves as the main tool for our analysis throughout this section.

\begin{proposition}
	\label{pro:linear_span_NCO}
Let $m\le n$ and $k<\frac{m}{2}$ be positive integers, and let $L_f>0$ and $L_A\ge0$. Suppose that we have a problem instance of \eqref{eq:NCOspecial}, called \emph{original instance}, where $X$ and $Y$ are compact Euclidean balls, $\vH\in\SS_+^n$ and $\|\vH\|\le L_f$, $\vA$ and $\vb$ are those in \eqref{eq:Ab}. Moreover, assume that $\vH$ satisfies $\vH\cK_{2s-1}\subseteq\cK_{2s}$ for any $s\le \frac k 2$ and $\vh\in\cK_0$, where $\cK_{i}$ is defined in \eqref{eq:KJ}.
	Then for any first-order method $\cM$ that is described by \eqref{eq:I}, there exists one other problem instance, called \emph{rotated instance},
	\begin{align}
	\label{eq:NCO_rotate}
	\tilde \phi^*:=\min_{\vx\in X}\Set{\tilde \phi(\vx):=\tilde f(\vx) + \max_{\vy\in Y}\langle \vb - \tilde\vA\vx, \vy\rangle},
	\end{align}
	such that $\|\tilde \vA\| = \|\vA\|$ and $\tilde f(\vx) =\frac{1}{2}\vx^\top\tilde\vH\vx -\vh^\top\vx$ with $\tilde\vH\in\SS_+^n$ and $\|\tilde\vH\| = \|\vH\|$. Specifically,
	\begin{align}
	\label{eq:NCO_UV}
	\tilde f(\vx) = f(\vU\vx)\text{ and }\tilde\vA = \vV^\top \vA\vU,
	\end{align}
	where $\vU$ and $\vV$ are certain orthogonal matrices such that $\vU\vh = \vh$ and $\vV\vb = \vb$. 
	In addition, 
	$(\vx^*,\vy^*)$ is a saddle point to the \emph{original instance} if and only if $(\hat\vx, \hat\vy):=(\vU^\top\vx^*,\vV^\top\vy^*)$ is a saddle point to the \emph{rotated instance}. 	
	Furthermore, when $\cM$ is applied to \eqref{eq:NCO_rotate}, for any $0\le t\le \frac{k-4}{2}$, its computed approximate solution $\xu^{(t)}$ satisfies
	\begin{align}
	\label{eq:linear_span_lb_NCO-1}
	& \tilde \phi(\xu^{(t)}) - \tilde \phi^* \ge \min_{\vx\in\cK_{k-1}} \phi(\vx) - \phi^* 
	\\
	& \|\xu^{(t)} - \hat\vx\|^2 \ge \min_{\vx\in\cK_{k-1}}\|\vx - \vx^*\|^2.
	\label{eq:linear_span_lb_NCO-2}
	\end{align}
	
\end{proposition}

\begin{proof}
For the given data, we apply Proposition \ref{pro:krylov} to obtain an instance $P(\vtheta; \vU^\top\vH\vU, -\vV^\top\vA\vU)$ with $\vtheta=(\vh, -\vb,X,Y,0)$ and orthogonal matrices $\vU$ and $\vV$ satisfying $\vU\vh=\vh$ and $\vV\vb=\vb$. Clearly, this gives the rotated instance in \eqref{eq:NCO_rotate}.
	
	Since $X$ and $Y$ are Euclidean balls, they are invariant under orthogonal transformation, and thus $\vU X=X$ and $\vV Y = Y$. By comparing \eqref{eq:NCOspecial} and \eqref{eq:NCO_rotate}, using \eqref{eq:NCO_UV} and the relations $\vU\vh = \vh$ and $\vV\vb = \vb$, it is not difficult to show that $(\vx^*,\vy^*)$ is a saddle point to the original instance if and only if $(\vU^\top\vx^*,\vV^\top\vy^*)$ is a saddle point to the rotated instance, and that $\tilde\phi^* = \phi^*$. 
	
	It remains to prove the inequalities in \eqref{eq:linear_span_lb_NCO-1} and \eqref{eq:linear_span_lb_NCO-2}. By Proposition \ref{pro:krylov}, when $\cM$ is applied to the rotated instance, the approximate solution $\xu^{(t)}\in \vU^\top \cK_{2t+3}$. Since $t\le \frac {k-4} {2}$, we have $2t+3\le k-1$, and hence from \eqref{eq:KJexpan}, it follows that $\vU\xu^{(t)}\in \cK_{2t+3}\subseteq \cK_{k-1}$. Therefore, we have the desired results from \eqref{eq:NCO_UV} and complete the proof.
\end{proof}

\subsection{Lower complexity bounds}
In this subsection, we construct ``hard'' instances of \eqref{eq:NCOspecial} and establish lower complexity bound results of first-order methods on solving \eqref{eq:saddle-prob}.  Let $m\le n$ and $k < \frac{m}{2}$ be positive integers, and let $L_f>0$ and $L_A\ge 0$. Our first instance is \eqref{eq:NCOspecial} with
$(\vH,\vh,\vA,\vb)$ given in \eqref{eq:HQh},  
$f(\vx)=\frac{1}{2}\vx^\top \vH\vx - \vh^\top\vx$
and
\begin{align}
\label{eq:XY_HQh}
X=\Set{\vx\in\RR^n|\|\vx\|^2\le R_X^2=k(2k+1)^2}, Y=\Set{\vy\in\RR^n|\|\vy\|^2\le R_Y^2={\textstyle\frac k 4}}.
\end{align} 
Clearly the above problem is a special instance of \eqref{eq:SPP}. In the following lemma, we give an optimal solution and the optimal objective value of the instance.
\begin{lemma}
	\label{lem:NCO_HQh}
	For \eqref{eq:NCOspecial} with
$(\vH,\vh,\vA,\vb)$ given in \eqref{eq:HQh} and $(X,Y)$ specified in \eqref{eq:XY_HQh}, it has an optimal solution $\vx^*$ given in \eqref{eq:sol-lam=0} and the optimal objective value
	$$
	\phi^* = -\left(\frac{L_f}{4} + \frac{L_A}{2\sqrt{2}}\right)k.
	$$
\end{lemma}

\begin{proof}
	The optimality condition (e.g., \cite{nemirovski2004prox}) for $\vx^*\in X$ to solve \eqref{eq:NCOspecial} is that there exists $\vy^*\in Y$ such that
	\begin{align}\label{eq:opt-cond-primal}
	\langle \nabla f(\vx^*) - \vA^\top \vy^*, \vx^*-\vx\rangle \le 0,\ \langle \vA\vx^*-\vb,\vy^*-\vy\rangle \le 0, \forall\, \vx\in X, \vy\in Y.
	\end{align}
Let $\vy^*$ be the vector given in \eqref{eq:ECO_HQh_sol-lam=0-y}. Note from the proof of Lemma \ref{lem:HQh}, it holds that $\nabla f(\vx^*)=\vH\vx^*-\vh=\vA^\top \vy^*$ and $\vA \vx^*-\vb=\vzero$. Hence, $(\vx^*,\vy^*)$ satisfies the above optimality condition. In addition, from \eqref{eq:ECO_Hh0_nrm-x-lam=0} and \eqref{eq:HQh_nrm-y-lam=0}, it follows that $\vx^*\in X$ and $\vy^*\in Y$. Therefore, $\vx^*$ is an optimal soluton, and it is straightforward to compute $\phi^*=\phi(\vx^*)= -(\frac{L_f}{4} + \frac{L_A}{2\sqrt{2}})k$. This completes the proof.
\end{proof}	

In the following lemma, we compute the minimum value of $\phi(\vx)$ over $\cK_{k-1}$. 

\begin{lemma}
	\label{lem:NCO_KkBound}
	Consider \eqref{eq:NCOspecial} with
$(\vH,\vh,\vA,\vb)$ given in \eqref{eq:HQh} and $(X,Y)$ specified in \eqref{eq:XY_HQh}. If $L_f\ge L_A$, then
	\begin{align}
	\label{eq:NCO_HQh_KkBound}
	\min_{\vx\in\cK_{k-1}}\phi(\vx) - \phi^* \ge \frac{L_fR_X^2}{16(2k+1)^2} + \left(\frac{\sqrt{2}+2}{4}\right)\frac{L_AR_XR_Y}{2k+1}.
	\end{align}
\end{lemma}
\begin{proof}
Since $Y$ is an Euclidean ball, it holds $$\max_{\vy\in Y}\langle \vb-\vA\vx,\vy\rangle = R_Y\|\vA\vx - \vb\|,$$ and thus the objective of  \eqref{eq:NCOspecial} is
	\begin{align}
	\label{eq:tmp9}
	\phi(\vx) = f(\vx) + R_Y\|\vA\vx - \vb\|.
	\end{align}
	From the first inequality in \eqref{eq:ineq-Ax-b-lem-Hh0}, we have
	\begin{equation}
	\label{eq:tmp7}
	\|\vA\vx-\vb\|^2\ge \frac{k L_A^2}{4} \overset{\eqref{eq:XY_HQh}} = \frac{L_A^2R_X^2}{4(2k+1)^2}.
	\end{equation}	
	In addition, note that $f(\vx)$ here is exactly the same as that discussed in Lemma \ref{lem:ECO_HQh_KkBounds}, and thus by \eqref{eq:ECO_HQh_fmin_K} we have that for any $\vx\in\cK_{k-1}$, 
	\begin{align}
	\label{eq:tmp8}
	f(\vx) \ge -\frac{1}{8}\left(L_f + \sqrt{2}L_A + \frac{L_A^2}{2L_f}\right)k.
	\end{align}
	Applying \eqref{eq:tmp7} and \eqref{eq:tmp8} to \eqref{eq:tmp9}, and noting the value of $\phi^*$ in Lemma \ref{lem:NCO_HQh}, we have for any $\vx\in\cK_{k-1}$ that
	\begin{align*}
	\phi(\vx) - \phi^* \ge & 
	\frac{1}{8}\left(L_f + \sqrt{2} L_A - \frac{L_A^2}{2L_f}\right)k + \frac{L_AR_XR_Y}{2(2k+1)}
	\\
	\ge & \frac{1}{8}\left(\frac{L_f}{2} + \sqrt{2} L_A\right)k + \frac{L_AR_XR_Y}{2(2k+1)},
	\end{align*}
	where the second inequality follows from the fact $L_f\ge L_A$. Rewriting the terms $L_fk$ and $L_Ak$ as 
	\begin{align*}
		L_fk \overset{\eqref{eq:XY_HQh}} = \frac{L_fR_X^2}{(2k+1)^2},\text{ and }L_Ak\overset{\eqref{eq:XY_HQh}} = L_A\sqrt{k}\cdot\sqrt{k} = L_A\cdot\frac{R_X}{2k+1}\cdot 2R_Y = \frac{2L_AR_XR_Y}{2k+1},
	\end{align*}
	we conclude the desired result in \eqref{eq:NCO_HQh_KkBound}.	
\end{proof}

Using Proposition \ref{pro:linear_span_NCO} and Lemma \ref{lem:NCO_KkBound}, we are able to show a lower complexity bound of first-order methods on \eqref{eq:saddle-prob} as summarized in the following theorem.

\begin{theorem}[lower complexity bound for SPPs]
	\label{thm:NCO_HQh_lb-complexity}
	Let $8<m\le n$ and $t < \frac{m}{4}-2$ be positive integers, $L_f>0$, and $L_A\ge0$.  Assume $L_f\ge L_A$.  Then for any first-order method $\cM$ described in \eqref{eq:I} on solving \eqref{eq:saddle-prob}, there exists a problem instance of \eqref{eq:saddle-prob} such that $\nabla f$ is $L_f$-Lipschitz continuous, $\|\vA\|=L_A$, and $X$ and $Y$ are Euclidean balls with radii $R_X$ and $R_Y$ respectively. In addition, 
	\begin{align}
	\label{eq:NCO_HQh_obj-err} 
	\phi( \xu^{(t)} )-\psi(\yu^{(t)})\ge & ~ \frac{L_fR_X^2}{16(4t+9)^2} + \left(\frac{\sqrt{2}+2}{4}\right)\frac{L_AR_XR_Y}{4t+9},
	\end{align}
	where $\phi$ and $\psi$ are the associated primal and dual objective functions, and $(\xu^{(t)},\yu^{(t)})$ is the approximate solution output by $\cM$.
\end{theorem}

\begin{proof} 
	Set $k=2t+4<\frac{m}{2}$ and consider the problem instance described above \eqref{eq:XY_HQh}. Note that the data $(\vH,\vh,\vA,\vb,X,Y)$ satisfies the conditions required by Proposition \ref{pro:linear_span_NCO}. Hence, we can apply the proposition to obtain a rotated instance \eqref{eq:NCO_rotate}. From \eqref{eq:linear_span_lb_NCO-1} and Lemma \ref{lem:NCO_KkBound}, it follows that
	\begin{align}\label{eq:ineq-phi-spp}
	\tilde \phi(\xu^{(t)}) - \tilde \phi^* \ge \min_{\vx\in\cK_{k-1}} \phi(\vx) - \phi^*\ge \frac{L_fR_X^2}{16(4t+9)^2} + \left(\frac{\sqrt{2}+2}{4}\right)\frac{L_AR_XR_Y}{4t+9}.
	\end{align}
Let $\tilde\psi$ be the dual objective function of \eqref{eq:NCO_rotate}. Then since $\yu^{(t)}\in Y$, it holds $\tilde\psi(\yu^{(t)}) \le \tilde\psi^* \le \tilde \phi^*$, where the second inequality follows from the weak duality. Therefore we have the desired result from \eqref{eq:ineq-phi-spp} and by dropping the tilde in the notation.	
\end{proof}

\begin{remark}\label{rmk:NCO_HQh_lb-complexity}
We make two remarks here. First, the lower bound in \eqref{eq:NCO_HQh_obj-err} has exactly the same form as the upper bound in \eqref{eq:rate-nesterov-spp}, and they differ only on the constants. Hence, the order of the convergence rate result in \eqref{eq:rate-nesterov-spp} is not improvable, and one can only improve that result by possibly decreasing the constants. Second, in Theorem \ref{thm:NCO_HQh_lb-complexity}, $L_f\ge L_A$ is assumed. Without the assumption, the first term on the right hand side of \eqref{eq:NCO_HQh_obj-err}  becomes less important because it will be dominated by the second term for big $t$ if $R_X$ and $R_Y$ are regarded as constants. In fact,  assuming $L_f$ and $L_A$ both positive but without $L_f\ge L_A$, we can consider the following SPP instance
$$\min_{\vx\in X}\max_{\vy\in Y} f(\vx)+\langle \vb-\vA\vx, \vy\rangle$$
with $f(\vx) = L_f\left(\frac{1}{2}x_k^2 + \frac{1}{2}\sum_{i=2k+1}^{n}x_i^2\right)$, $(\vA,\vb)$ given in \eqref{eq:Ab} and 
$$X=\Set{\vx\in\RR^n|\|\vx\|^2\le R_X^2=k(2k+1)^2}, Y=\Set{\vy\in\RR^n|\|\vy\|^2\le R_Y^2=\frac{4L_f^2}{L_A^2}k^3}.$$
It is easy to see that the associated primal problem is
$$\phi^*:=\min_{\vx\in X}\big\{\phi(x):=f(\vx)+R_Y\|\vA\vx-\vb\|\big\}.$$
Let $\vx^*$ and $\vy^*$ be respectively given in \eqref{eq:sol-lam=0} and \eqref{eq:sol-lam=0-y}. Then it is easy to verify the optimality conditions in \eqref{eq:opt-cond-primal} hold, and in addition $\vx^*\in X$ and $\vy^*\in Y$. Hence $\vx^*$ is the optimal primal solution, and the optimal primal objective $\phi^*=\frac{L_fk^2}{2}$. From the proof of Lemma \ref{lem:ECO_Hh0_KkBounds}, we have that for any $\vx\in\cK_{k-1}$, $f(\vx)=0$ and
$\|\vA\vx-\vb\|^2\ge\frac{kL_A^2}{4}$. Hence,
$$\min_{\vx\in\cK_{k-1}} \phi(\vx)-\phi^* \ge \frac{\sqrt{k}L_A R_Y}{2}-\frac{L_fk^2}{2} = \frac{\sqrt{k}L_A R_Y}{4} =\frac{L_A R_XR_Y}{4(2k+1)}\ge \frac{L_f R_X^2}{6(2k+1)}.$$
Therefore, using Proposition \ref{pro:linear_span_NCO}, we can show that for $t\le \frac{m}{4}-2$, the following lower complexity bound holds:
\begin{equation}\label{eq:NCO_Hh0_obj-err}
\phi(\xu^{(t)})-\psi(\yu^{(t)})\ge\max\Set{\frac{L_f R_X^2}{6(4t+9)},\,\frac{L_A R_XR_Y}{4(4t+9)}}\ge \frac{L_f R_X^2}{12(4t+9)} + \frac{L_A R_XR_Y}{8(4t+9)}.
\end{equation}
We leave the details to the interested readers.
\end{remark}

We finish this section by showing a lower complexity bound for SPPs when the function $f(\vx)$ in \eqref{eq:saddle-prob} is strongly convex.

\begin{theorem}[lower complexity bound for SPPs with strong convexity]
	Let $8<m\le n$ and $t < \frac{m}{4}-2$ be positive integers, and $\mu$ and $L_A$ be positive numbers. Then for any first-order method $\cM$ described in \eqref{eq:I}, there exists a problem instance of \eqref{eq:saddle-prob} such that $f$ is $\mu$-strongly convex, $\|\vA\|=L_A$, $X$ and $Y$ are Euclidean balls with radii $R_X$ and $R_Y$ respectively, and the associated primal problem \eqref{eq:SPP} has a unique optimal solution $\hat{\vx}\in X$. In addition, 
	\begin{equation}\label{eq:lbd-complexity-scvx-spp}
	\|\xu^{(t)}-\hat\vx\|^2 \ge \frac{5L_A^2 R_Y^2}{256\mu^2 (4t+9)^2}
	\end{equation}
	and
	\begin{align}
		\label{eq:lbd-complexity-scvx-spp-phi}
		\phi(\xu^{(t)}) - \psi(\yu^{(t)}) \ge \frac{5L_A^2 R_Y^2}{512\mu (4t+9)^2},
	\end{align}
	where $\phi$ and $\psi$ are the associated primal and dual objective functions, and $(\xu^{(t)},\yu^{(t)})$ is the approximate solution output by $\cM$.
\end{theorem}

\begin{proof}
	Set $k = 2t+4<\frac{m}{2}$ and consider the problem instance of \eqref{eq:NCOspecial}, where $f(\vx) = \frac{\mu}{2}\|\vx\|^2$, $\vA$ and $\vb$ are those in \eqref{eq:Ab}, and
	\begin{align}
	\label{eq:XY_scvx}
	\begin{aligned}
	&X=\Set{\vx\in\RR^n|\|\vx\|^2\le R_X^2:=k(2k+1)^2},\\[0.1cm]
	& Y=\Set{\vy\in\RR^n|\|\vy\|^2\le R_Y^2:=\frac{128\mu^2}{15L_A^2}k(k+1)^3(2k+1)}.
	\end{aligned}
	\end{align} 
	From the proof of Theorem \ref{thm:lb-complexity-lam=0-scvx}, it is easy to verify that $\vx^*$ in \eqref{eq:sol-lam=0} and $\vy^*$ in \eqref{eq:sol-lam=0-y-scvx} satisfy $\vx^*\in X$, $\vy^*\in Y$, and the optimality condition in \eqref{eq:opt-cond-primal}.
	Since $f$ is strongly convex, $\vx^*$ must be the unique optimal solution to \eqref{eq:NCOspecial}. From \eqref{eq:bd-iter-diff} and also the definitions of $X$ and $Y$ in \eqref{eq:XY_scvx}, it follows that
	\begin{equation}\label{eq:min-x-cK-k-1}
	\min_{\vx\in\cK_{k-1}}\|\vx-\vx^*\|^2
	\ge \frac{5L_A^2 R_Y^2}{256\mu^2 (2k+1)^2}.
	\end{equation}
	Note that the data in the considered instance satisfy all the conditions in Proposition \ref{pro:linear_span_NCO}. Hence, we can apply the proposition to obtain the corresponding rotated instance \eqref{eq:NCO_rotate}, which has a unique optimal solution $\hat{\vx}\in X$. Now use \eqref{eq:min-x-cK-k-1} and \eqref{eq:linear_span_lb_NCO-2} to obtain \eqref{eq:lbd-complexity-scvx-spp} by recalling $k=2t+4$.  
	Moreover, by the strong convexity and the optimality of $\hat\vx$, we have
	\begin{align*}
		\tilde \phi(\xu^{(t)}) - \tilde \phi^* 
		\ge \frac{\mu}{2}\|\xu^{(t)} - \hat\vx\|^2,
	\end{align*}
Let $\tilde\psi$ be the corresponding dual objective function. Together with \eqref{eq:lbd-complexity-scvx-spp} and the fact $\psi(\yu^{(t)})\le \tilde\psi^*\le\tilde\phi^*$, the above inequality gives \eqref{eq:lbd-complexity-scvx-spp-phi} by dropping the tilde in the notation. Therefore, we complete the proof.
\end{proof}

\section{On the tightness of the established lower complexity bounds}\label{sec:tightness}

In this section, we compare the established lower complexity bounds to the best known upper complexity bounds. It turns out that the lower complexity bounds developed in this paper are tight in terms of the order, and thus they can be used to justify the optimality of first-order methods in the literature.

\subsection{Upper complexity bounds of first-order methods on affinely constrained problems}
In \cite{xu2018accelerated-pdc}, a first-order primal-dual block coordinate update method is analyzed. As there is one single block, i.e., applied to \eqref{eq:ECO}, and the initial iterate $(\vx^{(0)},\vy^{(0)})=(\vzero,\vzero)$, the convergence rate result about the objective \cite[eqn.(20a)]{xu2018accelerated-pdc} is
\begin{equation*}
\big|f(\vx^{(t)})-f^*\big| \le \frac{1}{t}\left[\frac{L_f+\beta\|\vA\|^2}{2}\|\vx^*\|^2+\frac{\max\big(4\|\vy^*\|^2, (1+\|\vy^*\|)^2\big)}{2\beta}\right],
\end{equation*}
where $\beta$ is a Lagrangian penalty parameter used in the algorithm. Let\footnote{Although it is impractical to set $\beta$ in this way, we argue that we can find such a $\beta$ theoretically and thus the result in \eqref{eq:upper-bd-lalm} is a valid upper bound.}
$$\beta=\frac{\max\big(2\|\vy^*\|,1+\|\vy^*\|\big)}{\|\vA\|\cdot \|\vx^*\|}.$$
The above convergence rate result becomes
\begin{equation}\label{eq:upper-bd-lalm}
\big|f(\vx^{(t)})-f^*\big| \le \frac{1}{t}\left[\frac{L_f}{2}\|\vx^*\|^2+\|\vA\|\cdot \|\vx^*\|\max\big(2\|\vy^*\|,1+\|\vy^*\|\big)\right],
\end{equation}
which coincides with \eqref{eq:ECO_Hh0_obj-err-lam=0} if $\|\vy^*\|\ge 1$ except the difference of a constant multiple.

The work \cite{ouyang2015accelerated} proposes an accelerated linearized alternating direction method of multipliers (AL-ADMM). Applying to \eqref{eq:ECO}, i.e., setting one block to zero, we have from one convergence rate result in \cite[eqn.(2.34)]{ouyang2015accelerated} that
\begin{equation}\label{eq:t2-ouyang}
f(\vx^{(t)})-f^*\le \frac{2L_f D_X^2}{t(t+1)}+\frac{2\|\vA\|D_XD_Y}{t+1},
\end{equation}
where $D_X$ and $D_Y$ are the diameters of the primal and dual feasible sets. If the size of the optimal primal and dual solutions is assumed, then the above result coincides with that in \eqref{eq:ECO_HQh_obj-err-lam=0-gen} up to the difference of a constant multiple.

For the strongly convex case, the result in \eqref{eq:lbd-complexity-scvx-sp} indicates that given any $\vareps>0$, to have an iterate within $\sqrt\vareps$-neighborhood of $\vx^*$, the iterate number is at least
\begin{equation}\label{lbd-t-scvx}
t=\left\lceil\frac{\sqrt{5}L_A\|\vy^*\|}{32\mu\sqrt{\vareps}}-\frac{5}{2}\right\rceil,
\end{equation}
where $\lceil a\rceil$ denotes the smallest integer no less than $a\in \RR$. 
In \cite[proof of Thm.4]{xu2017iter-complexity-ialm}, it is shown that 
\begin{equation}\label{eq:result-xu2017-ialm}
f(\vx^{(t)})-f(\vx^*)+\langle\vy^*,\vA\vx^{(t)}-\vb\rangle\le \frac{\|\vy^*\|^2}{2\rho_0}+\vareps_0,
\end{equation}
where $(\vx^*,\vy^*)$ is a pair of primal-dual solution, and $\vx^{(t)}$ is the output of Nesterov's optimal first-order method applied to a penalized problem after $t$ iterations. In addition, with $\rho_0=\frac{2\|\vy^*\|^2}{\mu\vareps}$ and $\vareps_0=\frac{\mu\vareps}{4}$ in \eqref{eq:result-xu2017-ialm}, \cite[eqn.(49)]{xu2017iter-complexity-ialm} shows that the iteration number $t$ satisfies:
\begin{equation}\label{ubd-t-scvx}
t\le 2\left(\sqrt{\frac{L_f}{\mu}}+\frac{2L_A\|\vy^*\|}{\mu\sqrt{\vareps}}\right)\big(O(1)+\log\frac{1}{\vareps}\big).
\end{equation} 
From the strong convexity of $f$, it follows that
$$f(\vx^{(t)})-f(\vx^*)+\langle\vy^*,\vA\vx^{(t)}-\vb\rangle\ge\frac{\mu}{2}\|\vx^{(t)}-\vx^*\|^2$$
which together with \eqref{eq:result-xu2017-ialm} gives $\|\vx^{(t)}-\vx^*\|^2\le \vareps$. Hence, the dominant term in the upper bound \eqref{ubd-t-scvx} is the same as that in \eqref{lbd-t-scvx} except for a log term.

\subsection{Upper complexity bounds of first-order methods on saddle-point problems}
For optimization problems in the form of \eqref{eq:SPP}, a smoothing technique is proposed in \cite{nesterov2005smooth}. It first approximates the nonsmooth objective function by a smooth one and then minimizes the smooth approximation function by an accelerated gradient method. In \cite{nesterov2005smooth}, it is shown 
that, if $X$ and $Y$ are compact with diameter $D_X$ and $D_Y$ respectively, and the total number of iterations  is pre-specified to $t$, then the convergence rate  of this smoothing scheme applied to \eqref{eq:SPP} is given in \eqref{eq:rate-nesterov-spp}.
Comparing the upper bound in \eqref{eq:rate-nesterov-spp} and the lower bound in \eqref{eq:NCO_HQh_obj-err}, we conclude that our lower complexity bound in Theorem \ref{thm:NCO_HQh_lb-complexity} is tight in terms of the order, and that Nesterov's smooth scheme is an optimal method for computing approximate solutions to bilinear SPPs in the form of \eqref{eq:saddle-prob}. 

Note that Theorem \ref{thm:NCO_HQh_lb-complexity} also confirms the optimality of several follow-up works of \cite{nesterov2005smooth}. For example, when the algorithms in \cite{chen2014optimal,chen2017accelerated} are applied to solve \eqref{eq:SPP}, their convergence rates all coincide with the lower bound in  \eqref{eq:NCO_HQh_obj-err} up to a constant multiple, and hence these methods are all optimal first-order methods for solving problems in the form of \eqref{eq:SPP}. 

In the literature, there have also been several results on either the saddle point or the variational inequality formulation of \eqref{eq:SPP} \cite{nemirovski2004prox,monteiro2011complexity,monteiro2013iteration,he2012-rate-drs,chambolle2011first}. When applied to solve \eqref{eq:SPP} with $f\equiv 0$ (and hence $L_f\le L_A$), those results all imply

\begin{align*}
\phi(\vx^{(t)}) - \phi^*= O\left(\frac{L_AD_XD_Y}{t}\right),
\end{align*}
where $D_X$ and $D_Y$ are the diameters of $X$ and $Y$.
The above result indicates the tightness of the lower bound in \eqref{eq:NCO_Hh0_obj-err}.

\section{Concluding remarks}\label{sec:conclusion}

On finding solutions to bilinear saddle-point problems, we have established lower complexity bounds of first-order methods that acquire problem information through a first-order oracle and are described by a sequence of updating rules. Through designing ``hard'' instances of convex quadratic programming, we first show the lower complexity bound results under a linear span assumption on solving affinely constrained problems. Then by a rotation invariance technique, we extend the results to general first-order methods that are still applied to affinely constrained problems. Finally, we establish the results for general first-order methods on solving bilinear saddle-point problems with compact primal and dual feasible regions. The established lower complexity bounds have been compared to several existing upper bound results. The comparison implies the tightness of our bounds and optimality of a few first-order methods in the literature.

We conclude the paper with a few more remarks. First, note that for affinely constrained problems, the feasibility residual in none of our results depends on the objective; see \eqref{eq:ECO_HQh_feas-err-lam=0} and \eqref{eq:ECO_HQh_feas-err-lam=0-gen} for example. 
This is reasonable because we can choose not to use the objective gradient though the oracle \eqref{eq:O} provides such information. However, towards finding an optimal solution, the objective information must be used. All existing works (e.g., \cite{lan2016iteration-alm, xu2017accelerated, chen2017accelerated}) on primal-dual first-order methods have objective-dependent quantity in their upper bounds on the feasibility error. One interesting question is how to derive a lower complexity bound of the feasibility residual that depends on the constraint itself and also the objective. To achieve that, we would need to enforce a minimum portion of objective information to be used in the solution update. Second, a few existing works \cite{lan2016gradient,lan2016conditional,lan2016accelerated} have shown that if $\nabla f$ is much more expensive than matrix-vector multiplication $\vA\vx$ and $\vA^\top\vy$, it could be beneficial to skip computing $\nabla f$ at the cost of more $\vA\vx$ and/or $\vA^\top\vy$. This setting is different from what we have made. In \eqref{eq:O}, we assume that one inquiry of the first-order oracle will obtain gradient and matrix-vector multiplications simultaneously. In the future work, we will allow multiple oracles that can return separate pieces of information, and we will pursue the lower bound of each oracle inquiry to reach a solution with desired accuracy and also design optimal oracle-based algorithms. Thirdly, in all our established results, we do not pre-specify the size of $X$ and $Y$ but allow them to be determined in the designed instances. That is the key reason why we obtain a lower complexity bound that looks greater than existing upper bound, e.g., by comparing \eqref{eq:rate-nesterov-spp} and \eqref{eq:NCO_Hh0_obj-err}. It is interesting to design ``hard'' instances to establish similar lower complexity bound results, provided that $L_f, L_A$ and the diameters of $X,Y$ are all given. We leave this to the future work.

\bibliographystyle{abbrv}

\end{document}